\documentclass[12pt, twoside, english, headsepline]{article}

\usepackage[margin=30mm]{geometry}
\usepackage{tikz,graphicx}
\usetikzlibrary{arrows}
\usetikzlibrary{patterns}

\tikzset{
  mid arrow/.style={postaction={decorate,decoration={
        markings,
        mark=at position .5 with {\arrow[#1]{stealth}}
      }}},
}
\usetikzlibrary{decorations.markings} 
\tikzset
 {every pin/.style = {pin edge = {<-}}, 
  > = stealth, 
  flow/.style = 
   {decoration = {markings, mark=at position #1 with {\arrow{>}}},
    postaction = {decorate}
   },
  flow/.default = 0.5,   
  main/.style = {line width=1pt}
 }

\usepackage{amsmath,amsthm,amssymb,amsfonts, mathtools, amsbsy, dsfont}
\usepackage[utf8]{inputenx}
\usepackage{xcolor}
\usepackage{microtype}
\usepackage{hyperref}
\hypersetup{linkcolor=blue}

\newcommand*\dif{\mathop{}\!\mathrm{d}}  

\allowdisplaybreaks 

\newtheorem{theorem}{Theorem}[section] 
\newtheorem{corollary} {Corollary}[section] 
\newtheorem{proposition}{Proposition}[section] 

\theoremstyle{definition} 

\newtheorem{remark}{Remark}[section] 
\numberwithin{equation}{section} 

\providecommand{\keywords}[1]
{
  \small	
  \textbf{\textit{Keywords: }} #1
}
\providecommand{\MSC}[1]
{
  \small	
  \textit{2020 MSC: } #1   
}

\title{Interacting point processes}
\author{Fabrizio Cinque$^1$ and Enzo Orsingher$^2$\\
        \small Department of Statistical Sciences, Sapienza University of Rome, Italy \\
        \small $^1$fabrizio.cinque@uniroma1.it $^2$enzo.orsingher@uniroma1.it
}

\begin{document}

\maketitle

\begin{abstract}
We study two different types of vector point processes with interacting components, introducing a migration-type effect. The first case concerns two groups which modify their states with rate functions depending on time only. This yields a representation of the vector process in terms of independent non-homogeneous Skellam processes. In the general case, the decomposition involves independent Poisson processes. The second model is a birth-death-migration vector process. In the case of the linear death-migration we show that, for a fixed time instant, the vector  is equal in distribution to the sum of two independent Multinomial random variables. As a byproduct we derive the distribution of a pure migration process. Finally, we study the described vector processes time-changed with the inverse of Bernstein subordinator, establishing a general result concerning the relatioship between fractional difference-differential equations and the probability mass function of a wider class of point processes.
\end{abstract} \hspace{10pt}


\MSC{Primary 60G55; Secondary 60G22}



\section{Introduction}

In this paper we introduce two different types of vector point processes with interacting components, introducing a migration effect, one related to the Skellam-Poisson processes and one related to birth-death processes. Similar topics were deeply discussed by numerous remarkable authors who introduced stochastic versions of interacting growth models, see \cite{B1973, P1975, W1963} and the references therein. Recently, some results concerning point processes with immigration appeared in the papers \cite{SN2025, VK2024, VK2025} where the authors considered a one-dimensional birth-death process with some immigration policies. A relevant difference (both theoretically and mathematically) of our research lies on the fact that instead of assuming a migration from an outer source, also see \cite{B1970}, we assume that the groups can exchange their units. The interested reader can also find in the literature other recent evolution models, see for instance \cite{DcP2024} and the references therein.

Skellam processes first appeared in the paper \cite{S1946} and they have been recently generalized and studied in further detail by several authors, see \cite{BS2024, CO2025, GKL2020} and the references therein. We recall that, with $\mathcal{I}\subset \mathbb{R}\setminus\{0\}, \ |\mathcal{I}|<\infty$ and integrable $\lambda_i : [0,\infty] \longrightarrow [0, \infty)$ such that $\Lambda_i(t) = \int_0^t \lambda_i(s)\dif s<\infty,\ \forall\ t\ge0,\ i\in \mathcal{I}$, the probability generating function of a non-homogeneous generalized Skellam process $S$ is
\begin{equation}\label{generatriceProbabilitaSkellam}
	\mathbb{E} u^{S(t)} = \mathbb{E} u^{\sum_{i\in\mathcal{I}} i N_i(t)} = e^{-\sum_{i\in\mathcal{I}} \Lambda_i(t)(1-u^i) }.
\end{equation}
where in the first equality one uses the fundamental relationship $S(t) \stackrel{d}{=} \sum_{i\in\mathcal{I}} i N_i(t)$ with $N_i$ being independent Poisson processes with rate functions $\lambda_i$, see Theorem 2.1 of \cite{CO2025}.

These generalized Skellam processes contain as a particular case also the so-called Poisson process of order $K$ (also known as generalized counting process) introduced in \cite{P1984} and recently studied in a series of papers started with \cite{DcMM2016}. Generalized Skellam processes have good probabilistic properties and a wide spectrum of possible applications, ranging from insurance applications to modeling the intensity difference of pixels in cameras \cite{HJK2007} or the difference of the number of goals of two competing teams in a football game \cite{KN2008}.

In Section 2, we introduce a stochastic vector representing two groups which modify their states with rate functions depending on time only, see equation \eqref{definizioneInfinitesimaleCoppiaTipoSkellam}, therefore we have a structure resembling non-homogeneous dependent Skellam-type processes where the components may increase/decrease simultaneously or exchange units, that is a migration effect. Under the given assumptions, we prove that the stochastic vector can be decomposed into the sum of independent classical Skellam processes, see Theorem \ref{teoremaCoppiaSkellamTerminiSkellam}. This representation is extremely useful and completely characterizes the vector process. Moreover, it permits us to derive further properties and, in the homogeneous case, we can express the vector process as a compound vector Poisson process, see Proposition \ref{proposizioneSkellamInterazioneOmogeneoPoissonComposto}. Finally, we show that a Skellam decomposition holds even if we assume more interacting groups, see Section \ref{sottosezioneCasoGeneraleSkellamMultidimensionale}. On the other hand, in a generalized bi-dimensional setting where the two components can increase, decrease or exchange different amounts of units, a decomposition in terms of independent Poisson processes holds, see Theorem \ref{teoremaProcessoInterazioneSkellamGenerale}.
\\

Section 3 is devoted to the birth-death-migration vector process, where two birth-death population can exchange their units with a rate function depending on the number of elements in the departure group, as described in \eqref{definizioneInfinitesimaleProcessoNascitaMorteInterazione}. This kind of model may well prove to be capable of representing a collection of queues in service systems, epidemics and the evolution of interacting populations.
\\Birth-death processes have been studied since the early decades of the last century and presented in a general form by Kendall \cite{K1948}. In the last years birth-death processes have been the object of several works, see for instance \cite{OP2011, OPS2010} and recently a generalized version was developed in \cite{VK2024}.

After an overview on the birth-death-migration process, we focus on the case of death-migration processes and we prove that, for a fixed time $t$, the state of the vector can be represented by the sum of independent Multinomial random variables. This also permits us to derive results concerning the extinction probability. As a byproduct we obtain a pure migration process, where the two groups can only exchange units and therefore one component in sufficient to describe the vector (since the total number of units is constant in time).
\\

The last section concerns time-changed versions of the stochastic vector models analyzed in Sections 2 and 3. In detail, we consider Bernstein subordinator processes assuming the role of the time, which furnishes the system with a long memory property. Particular cases of this kind of subordination can be traced back to \cite{L2003} where the author derived a Poisson porcess time-changed with the inverse of a stable subordinator, connecting it to the Dzerbashyan-Caputo fractional operator. This theory widely developed and several papers applied it to different types of processes, among the others, we refer to \cite{AetAl2022, ALP2020, C2022, GOP2015, GK2023, PKS2011} and references therein. For papers focusing on Skellam-type processes see, for instance, \cite{CO2025, DcMM2016, GKL2020, KK2024} and for the birth-death processes see \cite{OPS2010, OP2011, VK2024}.

Here, we provide a general result regarding point processes time changed with the inverse of subordinators of Bernstein types. A simple application of this result permits us to derive a fractional difference-differential characterization of the time-changed versions of the vector process introduced in Sections 2 and 3 (as well as several results established in the cited papers).



\section{Interacting Skellam-type processes}\label{sezioneProcessiSkellamInterazione}

Let $(N_1,N_2) = \bigl\{\big (N_1(t),N_2(t)\big)\big\}_{t\ge0}$ be a stochastic vector process describing the evolution of two interacting groups of elements, taking values in $\mathbb{Z}^2$ for $t\ge0$. We assume that the initial composition of the vector is $(N_1(0), N_2(0)) = (n_1, n_2)\in\mathbb{Z}^2$ a.s., that the process has independent increments and the following infinitesimal transition probability mass function
\begin{align}\label{definizioneInfinitesimaleCoppiaTipoSkellam}
P\{N_1(t+\dif t) &= h+i,\, N_2(t+\dif t) = k+j\,|\,N_1(t) = h,\, N_2(t) = k\}  \\
& = \begin{cases}
\begin{array}{l l}
\lambda_1(t)\lambda_2(t)\dif t + o(\dif t), & i,j=1,\\
\mu_1(t)\mu_2(t)\dif t + o(\dif t), & i,j=-1,\\
\lambda_1(t)\mu_2(t) \dif t + \eta_{21}(t)\dif t + o(\dif t), & i=1,j=-1,\\
\mu_1(t)\lambda_2(t)\dif t  +\eta_{12}(t)\dif t + o(\dif t), & i=-1,j=1,\\
\lambda_1(t)\delta_2(t)\dif t + o(\dif t), & i=1,j=0,\\
\delta_1(t)\lambda_2(t)\dif t + o(\dif t), & i=0,j=1,\\
\mu_1(t)\delta_2(t)\dif t + o(\dif t), & i=-1,j=0,\\
\delta_1(t)\mu_2(t)\dif t + o(\dif t), & i=0,j=-1,\\
1-\big(\lambda_1(t)\lambda_2(t) + \mu_1(t)\mu_2 (t)\\
\ \ \ \ \ +\lambda_1(t)\mu_2 (t)+ \eta_{21}(t) + \mu_1(t)\lambda_2 (t)+\eta_{12}(t) \\
\ \ \ \ \ + \lambda_1(t)\delta_2(t) + \delta_1(t)\lambda_2(t) \\
\ \ \ \ \ + \mu_1(t)\delta_2(t) + \delta_1(t)\mu_2(t)\big) \dif t + o(\dif t), & i,j=0, \\
o(\dif t), & \text{otherwise}.
\end{array}
\end{cases} \nonumber
\end{align}
Thus, the groups can evolve by acquiring or losing one unit or by transferring one element from one group to the other. By definition of the transition probability mass both groups can modify their composition in the same moment, but at most one event (i.e. increasing, loss or transfer) can happen during each infinitesimal interval. In particular, the rate function $\eta_{12}$ concerns the movement of one element from $N_1$ to $N_2$ while $\eta_{21}$ represents the migration in the opposite direction.

Let us denote by $p_{h,k}(t) = P\{N_1(t) = h,N_2(t) = k\}, \ h,k\in\mathbb{Z}$ and $t\ge0$. From the infinitesimal transition probability mass in \eqref{definizioneInfinitesimaleCoppiaTipoSkellam}, with standard arguments, we derive the following difference-differential equation, for $t\ge0, h,k\in\mathbb{Z}$ (we omit the explicit dependence of the time for the rate functions),
\begin{align}
\frac{\dif p_{h,k}(t)}{\dif t} &= -(\lambda_1\lambda_2 + \mu_1\mu_2 +\lambda_1\mu_2 + \eta_{21} + \mu_1\lambda_2 +\eta_{12} + \lambda_1\delta_2 + \delta_1\lambda_2 + \mu_1\delta_2 + \delta_1\mu_2) p_{h,k}(t)\nonumber\\
& \ \ \ + \lambda_1\lambda_2 p_{h-1,k-1}(t)+ \mu_1\mu_2p_{h+1,k+1}(t) \nonumber\\
&\ \ \ + \lambda_1\delta_2 p_{h-1,k}(t)+ \delta_1\lambda_2 p_{h,k-1}(t)+ \mu_1\delta_2 p_{h+1,k}(t) + \delta_1\mu_2 p_{ h,k+1}(t)\nonumber\\
&\ \ \  +(\lambda_1\mu_2 + \eta_{21})p_{h-1+k+1}(t) + (\mu_1\lambda_2 +\eta_{12})p_{h+1,k-1}(t).\label{equazioneStatoProcessoSkellamInterazione}
\end{align}
By using the fact that $\sum_{h,k = -\infty}^\infty u^h v^k p_{h+i, k+j}(t) = u^{-i}v^{-j}\sum_{h,k=-\infty}^\infty u^h v^k p_{h,k}(t) =  u^{-i}v^{-j}G(t,u,v), $ for $i,j\in\mathbb{Z}$ and suitable $u,v\in [-1,1]$, from the previous equation we obtian that
\begin{align}
\frac{\partial G(t,u,v)}{\partial t} &= -(\lambda_1\lambda_2 + \mu_1\mu_2 +\lambda_1\mu_2 + \eta_{21} + \mu_1\lambda_2 +\eta_{12} + \lambda_1\delta_2 + \delta_1\lambda_2 + \mu_1\delta_2 + \delta_1\mu_2) G(t,u,v) \nonumber\\
&\ \ \ + \Biggl(\lambda_1\lambda_2 uv + \frac{\mu_1\mu_2}{uv}  +(\lambda_1\mu_2 + \eta_{21})\frac{u}{v} + (\mu_1\lambda_2 +\eta_{12})\frac{v}{u}  + \lambda_1\delta_2 u+ \delta_1\lambda_2 v\nonumber\\
&\ \ \ + \frac{\mu_1\delta_2}{u} + \frac{\delta_1\mu_2}{v}\Biggr) G(t, u,v).
\end{align}

By keeping in mind the initial conditions, we obtain that the probability generating function reads
\begin{align}
G(t,u,v)&= u^{n_1}v^{n_2} \text{exp}\Biggl(-\int_0^t  \biggl[\lambda_1(s)\lambda_2(s) (1-uv) + \mu_1(s)\mu_2(s)\Bigl(1-\frac{1}{uv}\Bigr)\nonumber\\
&\ \ \ +\big(\lambda_1(s)\mu_2(s) + \eta_{21}(s)\big)\Big(1-\frac{u}{v}\Big) + \big(\mu_1(s)\lambda_2(s) +\eta_{12}(s)\big)\Big(1-\frac{v}{u}\Big)\nonumber \\
&\ \ \ + \lambda_1(s)\delta_2(s)(1-u) + \delta_1(s)\lambda_2(s)(1-v)\nonumber\\
&\ \ \  + \mu_1(s)\delta_2(s)\Big(1-\frac{1}{u}\Big) + \delta_1(s)\mu_2(s)\Big(1-\frac{1}{v}\Big)\biggr]  \dif s\Biggr)\label{funzioneGeneratriceCoppiaTipoSkellam}\\
&=:  u^{n_1}v^{n_2} \text{exp}\Bigl(-\int_0^t R(s,u,v)\dif s\Bigr).\nonumber
\end{align}

The independence of the increments of the vector $(N_1, N_2)$ implies that the probability generating function of the increments is given by, with $0<s<t$, 
\begin{align}\label{funzioneGeneratriceIncrementiCoppiaTipoSkellam}
\mathbb{E} u^{N_1(t)-N_1(s)}v^{N_2(t)-N_2(s)} = \frac{G(t,u,v)}{G(s,u,v)} =  \text{exp}\Bigl(-\int_s^t R(r,u,v)\dif r\Bigr),
\end{align}
where in the last equality we used \eqref{funzioneGeneratriceCoppiaTipoSkellam}. From \eqref{funzioneGeneratriceIncrementiCoppiaTipoSkellam} we notice that the increments are stationary if and only if the rates functions are all constant (which implies that the function $R$ does not depend on the time variable $r$) and the initial values are $n_1=n_2=0$. Omitting the last condition, the shifted vector $(N_1-n_1,N_2-n_2)$ is stationary.

The joint probability generating function (\ref{funzioneGeneratriceCoppiaTipoSkellam}) yields the marginal probability generating function
\begin{align}
G_{1}(t,u)& = \mathbb{E}u^{N_1(t)} = G(t,u,1) \nonumber\\
&=\text{exp}\Biggl(-\int_0^t  \biggl[\Bigl(\lambda_1(s)\big(\lambda_2(s) +\mu_2(s)  +\delta_2(s)\big)+\eta_{21}(s)\Big) (1-u)\nonumber \\
& \ \ \ + \Bigl(\mu_1(s)\big(\lambda_2(s) +\mu_2(s) + \delta_2(s)\big)+\eta_{12}(s)\Big)\Bigl(1-\frac{1}{u}\Bigr)\biggr]\dif s\Biggr),\label{funzioneGeneratriceMarginaleCoppiaSkellam}
\end{align}
which is the probability generating function of a non-homogeneous Skellam process with rate functions $\bar\lambda_1 = \lambda_1(\lambda_2 +\mu_2  +\delta_2)+\eta_{21}, \bar\mu_1=\mu_1(\lambda_2 +\mu_2  +\delta_2)+\eta_{12}$, see \eqref{generatriceProbabilitaSkellam}.

Similarly, we obtain that the process $N_2$ is a non-homogeneous Skellam process with rate functions $\bar\lambda_2 = \lambda_2(\lambda_1 +\mu_1  +\delta_1)+\eta_{12}, \bar\mu_2=\mu_2(\lambda_1 +\mu_1  +\delta_1)+\eta_{21}$.  Thus, the interested reader can study in detail the properties of the marginal components $N_1,N_2$. For instance, the marginal probability mass function follows from the well-known result that a non-homogenoeus Skellam process at a fixed time $t$ is Skellam random variable and its distribution is expressed in terms of the modified Bessel functions $I_\nu(z) = \sum_{k=0}^\infty (z/2)^{2k+\nu}/ (k! \Gamma(k+\nu + 1))$, with $\nu\in \mathbb{R}$ and $x\in\mathbb{C}$. In detail, if $S$ is a non-homogeneous Skellam process with rate functions $\lambda_1$ and $\lambda_{-1}$ such that $\Lambda_1(t) = \int_0^t \lambda_1(s)\dif s, \Lambda_{-1}(t) = \int_0^t \lambda_{-1}(s)\dif s$ exist finite for each $t\ge0$ (we write $S\sim NHGSP (\lambda_{1}, \lambda_{-1})$), then
\begin{equation}\label{probabilitaMassaSkellamClassico}
	P\{S(t) = n\}= e^{-\bigl(\Lambda_1(t)+\Lambda_{-1}(t)\bigr)} \Biggl(\sqrt{\frac{\Lambda_1(t)}{\Lambda_{-1}(t)}}\Biggr)^n I_n\Bigl(2\sqrt{\Lambda_1(t)\Lambda_{-1}(t)}\Bigr), \ \ \ n\in\mathbb{Z}, \ t\ge0,
\end{equation}
which was proved by Skellam in \cite{S1946} in the homogenous case. We refer to \cite{CO2025, GKL2020, S1946} and references therein for further results on Skellam-type processes.

The next statement provides a Skellam decomposition of the vector $(N_1, N_2)$.

\begin{theorem}\label{teoremaCoppiaSkellamTerminiSkellam}
Let $(N_1,N_2)$ be the stochastic vector defined in \eqref{definizioneInfinitesimaleCoppiaTipoSkellam}. Then,
\begin{equation}\label{decomposizioneCoppiaSkellam}
N_1 \stackrel{d}{=} S_1+S_3+S_4 \ \text{ and }\ N_2 \stackrel{d}{=} S_2+S_3-S_4,
\end{equation}
 where the terms are independent non-homogeneous (classical) Skellam processes
\begin{align}
&S_1\sim NHSP(\lambda_1\delta_2, \mu_1\delta_2),\ S_1(0) = n_1, \ \ \ S_2\sim NHSP(\delta_1\lambda_2,\delta_1\mu_2),\ S_2(0) = n_2,\nonumber\\
&S_3\sim NHSP(\lambda_1\lambda_2, \mu_1\mu_2),\ S_3(0) =0,\ \ \ S_4\sim NHSP(\lambda_1\mu_2+\eta_{21}, \mu_1\lambda_2+\eta_{12}),\ S_4(0) = 0.\label{processiSkellamDecomposizioneSkelamInterazione}
\end{align}
\end{theorem}

\begin{proof}
It is sufficient to prove that the joint probability generating function of the right-hand sides of \eqref{decomposizioneCoppiaSkellam} coincides with \eqref{funzioneGeneratriceCoppiaTipoSkellam}. For $t\ge0$ and suitable $u,v$,
\begin{align}
\mathbb{E}u^{S_1(t)+S_3(t)+S_4(t)}v^{S_2(t)+S_3(t)-S_4(t)} &= \mathbb{E}u^{S_1(t)} \mathbb{E} v^{S_2(t)} \mathbb{E} (uv)^{S_3(t)} \mathbb{E} \Big(\frac{u}{v}\Big)^{S_4(t)}
\end{align}
which yields \eqref{funzioneGeneratriceCoppiaTipoSkellam} by keeping in mind \eqref{processiSkellamDecomposizioneSkelamInterazione} and introducing the probability generating function of the Skellam processes in  \eqref{generatriceProbabilitaSkellam}.
\end{proof}

In the decomposition of Theorem \ref{teoremaCoppiaSkellamTerminiSkellam}, the Skellam processes $S_1$ and $S_2$ count the jumps when only $N_1$ or $N_2$, respectively, changes. The process $S_3$ concerns the jumps when the components of the vector change with a concordant sign. The process $S_4$ concerns the jumps when the components of the vector change with a non-concordant sign or a migration occurs. In particular, the process $S_4$ may also be decomposed in the sums of two independent Skellam processes, one concerning the jumps with non-concordant signs (having rate functions $\lambda_1\mu_2$ and $\mu_1\lambda_2$) and one concerning the migrations (having rate functions $\eta_{21}$ and $\eta_{12}$).
\\

Thanks to Theorem \ref{teoremaCoppiaSkellamTerminiSkellam} one can study further properties of the stochastic vector process $(N_1,N_2)$ as shown in the following remarks.

\begin{remark}\label{remarkConteggioEventiSkellamInterazione}
Keeping in mind that a Skellam process can be decomposed as the difference of two independent Poisson processes, the decomposition in Theorem \ref{teoremaCoppiaSkellamTerminiSkellam} also permits us to study processes which count the number of "events" occurring, for instance,
\begin{equation}\label{processoConteggioVariazioniSkellamInterazione}
 N(t) =\big|\big\{0<s\le t\,:\, \big(N_1(s-\dif s),N_2(s-\dif s)\big) \not=\big(N_1(s),N_2(s)\big)\big\}\big|,\ \ \ t\ge0,
\end{equation}
where $|A|$ denotes the number of elements of the set $A$. The stochastic process $N$ counts the number of times that the vector $(N_1, N_2)$ modifies its state. With \eqref{decomposizioneCoppiaSkellam} at hand, $N(t)$ can be expressed as the 
$$ N(t) = \sum_{i=1}^4  \big|\{0<s\le t\,:\, S_i(s-\dif s) \not=S_i(s)\}\big|$$
and each term of the sum is an independent non-homogeneous Poisson process with rate function given by the sum of the rate functions of the corresponding Skellam process, see Remark 2.3 of \cite{CO2025}. Therefore $N$ is a non-homogeneous Poisson process with rate function $\bar\lambda = \lambda_1\lambda_2 + \mu_1\mu_2 +\lambda_1\mu_2 + \eta_{21} + \mu_1\lambda_2 +\eta_{12} + \lambda_1\delta_2 + \delta_1\lambda_2 + \mu_1\delta_2 + \delta_1\mu_2$.

In light of a similar argument, each type of events is counted by a non-homogeneous Poisson process. For instance, the number of migrations from $N_1$ to $N_2$ behaves as a Poisson process with rate function $\eta_{12}$.
\hfill $\diamond$
\end{remark}

\begin{remark}[Covariance]
We derive the covariance between the two components, for $0\le s\le t$,
\begin{align}
\text{Cov}\Bigl(N_1(s), N_2(t)\Big)& = \text{Cov} \Big(S_1(s)+S_3(s)+S_4(s),S_2(t)+S_3(t)-S_4(t)\Big)\nonumber\\
& =  \text{Cov} \Big(S_3(s),S_3(t)\Big) -  \text{Cov} \Big(S_4(s),S_4(t)\Big) \nonumber\\
& = \mathbb{V}S_3(s) - \mathbb{V}S_4(s) \nonumber\\
& = \int_0^s\Bigl[\big(\lambda_1(r)-\mu_1(r)\big)\big(\lambda_2(r) - \mu_2(r)\big)-(\eta_{12} + \eta_{21})\Big]\dif r. \label{covarianzaComponentiCoppiaSkellam}
\end{align}
In the above steps we used formulas (2.9) and (2.11) of \cite{CO2025}, where the reader can also find further details on the autocovariance of the non-homogeneous Skellam process. Formula \eqref{covarianzaComponentiCoppiaSkellam} also implies that $\text{Cov}\big(N_1(s), N_2(t)\big) = \text{Cov}\big(N_1(s), N_2(s)\big)$. Note that, as expected, the autocovariance of $S_3$ (which coincides with its variance at the lower time instant) increases the covariance between $N_1$ and $N_2$ since it concerns their concordant jumps, while the autocovariance of $S_4$ has the opposite effect since it concerns the non-concordant jumps and the migrations.
\hfill $\diamond$
\end{remark}

\begin{remark}
With the Skellam distribution \eqref{probabilitaMassaSkellamClassico} at hand, Theorem \ref{teoremaCoppiaSkellamTerminiSkellam} permits us to write an explicit form of the joint probability mass of $\big(N_1(t),N_2(t)\big)$. Let us denote the integral of the rate functions as follows, with $t\ge0$, $\Lambda_1^+(t) = \int_0^t\lambda_1(s)\delta_2(s)\dif s$, $\Lambda_1^-(t) = \int_0^t  \mu_1(s)\delta_2(s)\dif s$, $\Lambda_2^+(t) = \int_0^t \delta_1(s)\lambda_2(s)\dif s$, $\Lambda_2^-(t) = \int_0^t \delta_1(s)\mu_2(s)\dif s$, $\Lambda_3^+(t) = \int_0^t \lambda_1(s)\lambda_2(s)\dif s, \Lambda_3^-(t) = \int_0^t \mu_1(s)\mu_2(s)\dif s$, $\Lambda_4^+(t) = \int_0^t \big(\lambda_1(s)\mu_2(s)+\eta_{21}(s)\big)\dif s, \Lambda_4^-(t) = \int_0^t \big(\mu_1(s)\lambda_2(s)+\eta_{12}(s)\big)\dif s$. Thus, for $t\ge0$ and $m,n\in \mathbb{Z}$,
\begin{align}
&P\{N_1(t) = m,N_2(t)=n\} \\
& = \sum_{h,k = 0}^\infty P\{S_1 = h\} P\{S_2 = k\} P\{S_3+S_4 = m-h, S_3-S_4 = n-k\} = \nonumber \\
& = e^{-\Lambda_1^+(t)-\Lambda_1^-(t) - \Lambda_2^+(t)-\Lambda_2^-(t) -\Lambda_3^+(t)-\Lambda_3^-(t) -\Lambda_4^+(t)-\Lambda_4^-(t) }\nonumber\\
&\ \ \ \times \sum_{h,k = 0}^\infty \Bigg(\frac{\Lambda_1^+(t)}{\Lambda_1^-(t)}\Bigg)^{\frac{h}{2}} \Bigg(\frac{\Lambda_2^+(t)}{\Lambda_2^-(t)}\Bigg)^{\frac{k}{2}}  \Bigg(\frac{\Lambda_3^+(t)}{\Lambda_3^-(t)}\Bigg)^{\frac{m-h+n-k}{2}} \Bigg(\frac{\Lambda_4^+(t)}{\Lambda_4^-(t)}\Bigg)^{\frac{m-h-n+k}{2}} \nonumber \\
&\ \ \ \times I_{h}\Big(2 \sqrt{\Lambda_1^+(t)\Lambda_1^-(t)}\Big) I_{k}\Big(2 \sqrt{\Lambda_2^+(t)\Lambda_2^-(t)}\Big) I_{\frac{m-h+n-k}{2}}\Big(2 \sqrt{\Lambda_3^+(t)\Lambda_3^-(t)}\Big) I_{\frac{m-h-n+k}{2}}\Big(2 \sqrt{\Lambda_4^+(t)\Lambda_4^-(t)}\Big). \nonumber
\end{align}
\hfill$\diamond$
\end{remark}

\begin{remark}
Theorem \ref{teoremaCoppiaSkellamTerminiSkellam} implies that a linear combination of the marginal processes $N_1$ and $N_2$ is a non-homogeneous generalized Skellam process. Indeed, for $a,b\in\mathbb{R}$,
\begin{align*}
&aN_1+bN_2 \\
&= aS_1 + bS_2 +(a+b)S_3 + (a-b)S_4 \sim NHGSP\Big(\{\pm a,\pm b, \pm(a+b), \pm(a-b)\},\\
&\ \ \  (\lambda_1\delta_2,\mu_1\delta_2, \delta_1\lambda_2,\delta_1\mu_2, \lambda_1\lambda_2, \mu_1\mu_2, \lambda_1\mu_2+\eta_{21}, \mu_1\lambda_2+\eta_{12}) \Big).
\end{align*}
where we use the notation introduced in Definition 2.1 of \cite{CO2025}. Interesting examples are $N_1+N_2 = S_1+S_2+2S_3$ and $N_1-N_2 = S_1-S_2+2S_4,$ which are non-homogeneous Skellam processes of order $2$, see \cite{CO2025}.
\hfill$\diamond$
\end{remark}

\begin{proposition}\label{proposizioneSkellamInterazioneOmogeneoPoissonComposto}
Let $(N_1,N_2)$ be a homogeneous version of the vector process described in \eqref{definizioneInfinitesimaleCoppiaTipoSkellam}. Let us denote by $\bar\lambda = \lambda_1\lambda_2 + \mu_1\mu_2 +\lambda_1\mu_2 + \eta_{21} + \mu_1\lambda_2 +\eta_{12} + \lambda_1\delta_2 + \delta_1\lambda_2 + \mu_1\delta_2 + \delta_1\mu_2$. Then,
\begin{equation}\label{coppiaSkellamPoissonComposto}
\big(N_1(t),N_2(t)\big) \stackrel{d}{=} \sum_{k = 1}^{N(t)} (X_k,Y_k),\ \ \ t\ge0,
\end{equation}
where $N$ is a homogeneous Poisson process with rate $\bar\lambda$ and $(X_k,Y_k)$ are independent copies of
\begin{align}\label{definizioneComponentiPoissonCompostoSkellam}
(X,Y) = \begin{cases}
\begin{array}{l l}
(1,1), & \ \lambda_1\lambda_2/\bar\lambda,\\
(-1,-1), &\  \mu_1\mu_2/\bar\lambda,\\
(1,-1), & \ (\lambda_1\mu_2+\eta_{21})/\bar\lambda,\\
(-1,1), & \ (\mu_1\lambda_2+\eta_{12})/\bar\lambda,\\
(1,0), &\ \lambda_1\delta_2/\bar\lambda,\\
(0,1), & \ \delta_1\lambda_2/\bar\lambda,\\
(-1, 0), & \ \mu_1\delta_2/\bar\lambda,\\
(0, -1), & \ \delta_1\mu_2/\bar\lambda.\\
\end{array}
\end{cases} 
\end{align}
\end{proposition}

\begin{proof}
It is straightforward to show that the probabity generating function of the right-hand side of \eqref{coppiaSkellamPoissonComposto} coincides with \eqref{funzioneGeneratriceCoppiaTipoSkellam}. Indeed, the probability generating function of a vector compound Poisson reads
\begin{align}
\mathbb{E} u^{\sum_{k=1}^{N(t)} X_k} v^{\sum_{k=1}^{N(t)} Y_k} = \text{exp}\Bigl(-\bar\lambda t\big(1-\mathbb{E}u^Xv^Y\big)\Big),
\end{align}
and the result follows by inserting the probability generating function of $(X,Y)$, which emerges from \eqref{definizioneComponentiPoissonCompostoSkellam}. 
\end{proof}

\subsection{Generalizations}

\subsubsection{Multiple interacting processes}\label{sottosezioneCasoGeneraleSkellamMultidimensionale}

We consider a three-dimensional version of the model defined in \eqref{definizioneInfinitesimaleCoppiaTipoSkellam}, that is a model with three different stochastic groups $(N_1,N_2,N_3)$. According to the model above, we restrict ourselves under the hypothesis that in each group it can occur just one "event" per instant of time, meaning that a group cannot, at the same instant of time, either take part into two migrations (for instance $N_1$ cannot give to $N_2$ and simultaneously receive from $N_3$) or add/lose an element and receive/give another element because of migration. This is a sufficient, but not necessary, condition to have groups which may increase or decrease of maximum one unit per instant of time. Under this hypothesis, proceeding as shown in the case above, the marginal components are dependent non-homogeneous (classical) Skellam processes which can be decomposed into a linear combination of some other (classical) Skellam processes representing the different possible outcomes occurring at each instant of time:
\begin{align}
&N_1 \stackrel{d}{=} S_1+ S_{1,2} + S_{1,-2} + S_{1,3} + S_{1,-3} + S_{1,2,3} + S_{1,-2,3} + S_{1,2,-3}+ S_{1,-2,-3}\nonumber\\
&N_2 \stackrel{d}{=}  S_2 + S_{1,2} - S_{1,-2} + S_{2,3}  + S_{2,-3} + S_{1,2,3} - S_{1,-2,3} + S_{1,2,-3} - S_{1,-2,-3}\nonumber\\
&N_3 \stackrel{d}{=} S_3 + S_{1,3} - S_{1,-3} + S_{2,3} - S_{2,-3} + S_{1,2,3} + S_{1,-2,3} - S_{1,2,-3} -S_{1,-2,-3},\label{decomposizioneSkellamInterazionePiuGruppi}
\end{align}
 with $S_1\sim NHSP(\lambda_1\delta_2\delta_3,\mu_1\delta_2\delta_3),$ $S_{1,2}\sim NHSP(\lambda_1\lambda_2\delta_3,\mu_1\mu_2\delta_3),$  $S_{1,-2}\sim NHSP(\lambda_1\mu_2\delta_3 + \eta_{21}\delta_3,\mu_1\lambda_2\delta_3 + \eta_{12}\delta_3), \dots,S_{2,-3} \sim NHSP(\delta_1\lambda_2\mu_3+\delta_1\eta_{32},\delta_1\mu_2\lambda_3 + \delta_1\eta_{23}),$ $ S_{1,2,3}\sim NHSP(\lambda_1\lambda_2\lambda_3,\mu_1\mu_2\mu_3),$ $S_{1,-2,3}\sim NHSP(\lambda_1\mu_2\lambda_3 + \eta_{21}\lambda_3 + \lambda_1\eta_{23}, \mu_1\lambda_2\mu_3 + \eta_{12}\mu_3 + \mu_1\eta_{32}),\dots, S_{1,-2,-3}\sim NHSP(\lambda_1\mu_2\mu_3+\eta_{21}\mu_3 +\eta_{31}\mu_2, \mu_1\lambda_2\lambda_3+\eta_{12}\lambda_3+\eta_{13}\lambda_2)$. The reader can recognize the role of each term analyzing the rate functions (also consider the subscripts of the components, for instance, $S_{1,-2}$ concerns the events when $N_1$ increases and $N_2$ decreases).

The interested reader can observe that the decomposition \eqref{decomposizioneSkellamInterazionePiuGruppi} yields suitable versions of the remarks derived as consequences of the Theorem \ref{teoremaCoppiaSkellamTerminiSkellam}. We note that the linear combinations $aN_1+bN_2+cN_3$, with a $a,b,c\in\{\pm1\}$, are Skellam processes of order $3$.
\\

In general, a model with $n$ interacting groups, under the hypothesis that for each group can occur just one "event" (as above), can be expressed in terms of the combination of $(3^n-1)/2$ independent non-homogeneous (classical) Skellam processes, still having dependent marginal components of Skellam-type.

\subsubsection{Bidimensional generalized Skellam-type interacting processes}

We can extend the results presented above to a more general interacting vector process, where the two components can increase, decrease and exchange different amounts of units. Let the following sets describe the possible changes of the two groups,
\begin{align}
&\mathcal{I}_1 =\{k\in\mathbb{Z}\,:\, N_1 \text{ modifies of $k$ units, not depending on $N_2$}\},\nonumber\\
&\mathcal{I}_2 =\{k\in\mathbb{Z}\,:\, N_2 \text{ modifies of $k$ units, not depending on $N_1$}\},\nonumber\\
&\mathcal{N}_1 =\{k\in\mathbb{N}\,:\, N_1 \text{ forwards $k$ units to $N_2$}\},\nonumber\\
&\mathcal{N}_2 =\{k\in\mathbb{N}\,:\, N_2 \text{  forwards $k$ units to $N_1$}\}, \label{insiemiMidificheStatoVettoreInterazioneSkellamGeneralizzato}
\end{align}
with $|\mathcal{I}_1|, |\mathcal{I}_2|, |\mathcal{N}_1|, |\mathcal{N}_2|<\infty$. Also assume the integrable rate functions $\lambda^1_i,\lambda_j^2>0$ for $i\in\mathcal{I}_1, j\in \mathcal{I}_2$ and equal to $0$ otherwise, $\eta^{12}_i,\eta_j^{21}>0$ for $i\in\mathcal{N}_1, j\in\mathcal{N}_2$ and equal to $0$ otherwise, and $\delta^1,\delta^2>0$.

Then, we define a stochastic vector process $(N_1,N_2)$ such that $\big(N_1(0),N_2(0)\big) = (n_1,n_2)\in\mathbb{Z}^2$ a.s., with independent increments and having the following infinitesimal behavior $t\ge0,\ h,k\in\mathbb{Z}$
\begin{align}\label{definizioneInfinitesimaleCoppiaTipoSkellamGeneralizzato}
P\{N_1&(t+\dif t) = h+i,\, N_2(t+\dif t) = k+j\,|\,N_1(t) = h,\, N_2(t) = k\}  \\
& = \begin{cases}
\begin{array}{l l}
\lambda_i^1(t)\lambda_j^2(t)\dif t + o(\dif t), & i\not=-j,\ (i,j)\in\mathcal{I}_1\times \mathcal{I}_2 \\ & \text{ or } i=-j,\ |i|\not \in\mathcal{N}_1,|j|\not\in\mathcal{N}_2,\\
\lambda_{i}^1(t)\lambda_{-i}^2(t) \dif t + \eta_{|i|}^{12}(t)\dif t + o(\dif t), & j=-i\in\mathcal{N}_1\ (i<0),\\
\lambda_{-j}^1(t)\lambda_{j}^2(t) \dif t + \eta_{|j|}^{21}(t)\dif t + o(\dif t), & i=-j\in\mathcal{N}_2\ (j<0),\\
\lambda_i^1(t)\delta^2(t)\dif t + o(\dif t), & i\in\mathcal{I}_1,j=0,\\
\delta^1(t)\lambda_j^2(t)\dif t + o(\dif t), & i=0,j\in\mathcal{I}_2,\\
1-\Big( \sum_{i\in\mathcal{I}_1,j\in\mathcal{I}_2} \lambda_i^1(t)\lambda_j^2(t) \\
\ \ \ \ \ +\sum_{i\in\mathcal{I}_1} \lambda_i^1(t)\delta^2(t)+\sum_{j\in\mathcal{I}_2} \delta_i^1(t)\lambda_j^2(t)\\
\ \ \ \ \ + \sum_{i\in\mathcal{N}_1} \eta_{i}^{12}(t) + \sum_{j\in\mathcal{N}_2} \eta_{j}^{21}(t)\Big) \dif t + o(\dif t), & i,j=0, \\
o(\dif t), & \text{otherwise}.
\end{array}
\end{cases} \nonumber
\end{align}
where in the first case each component changes but it is certain that there is no migration; the second case describes when the change can occur because of the migration of units from $N_1$ to $N_2$; the third case describes when the change can occur because of the migration of units from $N_2$ to $N_1$; the fourth and fifth cases describe when just one group changes.

Note that the restriction of having jumps of integer size can be relaxed by modifying the sets in \eqref{insiemiMidificheStatoVettoreInterazioneSkellamGeneralizzato} and the domain of the variables $h,k$ in \eqref{definizioneInfinitesimaleCoppiaTipoSkellamGeneralizzato}.

\begin{theorem}\label{teoremaProcessoInterazioneSkellamGenerale}
Let $(N_1,N_2)$ be the stochastic vector process defined in \eqref{definizioneInfinitesimaleCoppiaTipoSkellamGeneralizzato}. Then, the marginal components are non-homogeneous generalized Skellam processes given by
\begin{align}
&N_1 = \sum_{h\in\mathcal{I}_1} h N_h^{1,\mathcal{I}} + \sum_{h\in\mathcal{I}_1} h \sum_{k\in\mathcal{I}_2} N_{h,k}^{\mathcal{I}} +\sum_{h\in\mathcal{N}_1} (-h) N_h^{1,\mathcal{N}} + \sum_{k\in\mathcal{N}_2} k N_k^{2,\mathcal{N}},\label{rappresentazioneInterazioneSkellamGeneralizzatoProcessiPoisson}\\
&N_2 = \sum_{k\in\mathcal{I}_2} k N_k^{2,\mathcal{I}} + \sum_{k\in\mathcal{I}_2} k \sum_{h\in\mathcal{I}_1} N_{h,k}^{\mathcal{I}} +\sum_{h\in\mathcal{N}_1} h N_h^{1,\mathcal{N}} + \sum_{k\in\mathcal{N}_2} (-k) N_k^{2,\mathcal{N}}, \label{rappresentazioneInterazioneSkellamGeneralizzatoProcessiPoissonDue}
\end{align}
where
\begin{align}
&N_h^{1,\mathcal{I}} \sim PP(\lambda^1_h\delta^2),\ h\in\mathcal{I}_1,\ \ \ N_k^{2,\mathcal{I}} \sim PP(\delta^1\lambda^2_k),\ k\in\mathcal{I}_2, \ \ \ N_{h,k}^{\mathcal{I}}\sim PP(\lambda_h^1\lambda_k^2),\ (h,k)\in\mathcal{I}_1\times\mathcal{I}_2,\nonumber\\
&N_h^{1,\mathcal{N}}\sim PP(\eta_h^{12}),\ h\in\mathcal{N}_1,\ \ \ N_k^{2,\mathcal{N}}\sim PP(\eta_k^{21}),\ k\in\mathcal{N}_2,
\end{align}
are independent non-homogeneous Poisson processes (PP).
\end{theorem}

\begin{proof}
The proof follows by suitably adapting the arguments used to derive the probability generating function \eqref{funzioneGeneratriceCoppiaTipoSkellam} and Theorem \ref{teoremaCoppiaSkellamTerminiSkellam}.
\end{proof}

\begin{remark}[Interacting Skellam processes of order $K$]
If we assume that the components of the vector process can increase/decrease or exchange up to $K\in\mathbb{N}$ units, then $\mathcal{I}_1 = \mathcal{I}_2 = \{-K,\dots,-1,1,\dots,K\}$ and $\mathcal{N}_1 = \mathcal{N}_2 = \{1,\dots, K\}$. Therefore, the representation \eqref{rappresentazioneInterazioneSkellamGeneralizzatoProcessiPoisson} reduces to
\begin{align*}
	&N_1 = \sum_{\substack{h=-K\\h\not = 0}}^K h N_h^1 + \sum_{\substack{h=-K\\h\not = 0}}^K h \sum_{\substack{k=-K\\h\not = 0}}^K  N_{h,k} - \sum_{h=1}^K h N_h^{1,\mathcal{N}} +\sum_{k=1}^{K} k N_k^{2,\mathcal{N}}  = S_1^{K} + \sum_{\substack{h,k=-K\\h,k\not = 0}}^K h N_{h,k} + S^{\mathcal{N}},\nonumber\\
\end{align*}
and similarly, \eqref{rappresentazioneInterazioneSkellamGeneralizzatoProcessiPoissonDue} turns into
\begin{equation*}
	N_2 = S_2^K + \sum_{\substack{h,k=-K\\h,k\not = 0}}^K k N_{h,k} - S^{\mathcal{N}},
\end{equation*}
where $S^K_{1} = \sum_{h=1}^K h\big(N_h^{1,\mathcal{I}}- N_{-h}^{1,\mathcal{I}}\big),\  S_2^K = \sum_{h=1}^K h\big(N_h^{2,\mathcal{I}}- N_{-h}^{2,\mathcal{I}}\big),\ S^{\mathcal{N}} = \sum_{h=1}^K h\big(N_h^{2,\mathcal{N}} - N_h^{1,\mathcal{N}} \big) $ are independent non-homogeneous Skellam processes of order $K$, see \cite{CO2025} and the references therein.
\hfill $\diamond$
\end{remark}

\section{Interacting birth-death processes}\label{sezioneProcessiNascitaMorteMigrazione}

Let $(N_1,N_2) = \bigl\{\big (N_1(t),N_2(t)\big)\big\}_{t\ge0}$ be a stochastic vector process describing the evolution of two interacting birth-death processes (which we also call \textit{birth-death-migration vector process}). Let us assume the vector has independent increments, $\big (N_1(0),N_2(0)\big) = (n_1,n_2)\in \mathbb{N}^2$ a.s. and the following infinitesimal transition probability mass function, with $h,k\in\mathbb{N}_0$,
\begin{align}\label{definizioneInfinitesimaleProcessoNascitaMorteInterazione}
P\{&N_1(t+\dif t)= h+i,\, N_2(t+\dif t) = k+j\,|\,N_1(t) = h,\, N_2(t) = k\}  \\
& = \begin{cases}
\begin{array}{l l}
\lambda_1(t)h \dif t + o(\dif t), & i=1,j=0,\\
\lambda_2(t)k\dif t + o(\dif t), & i=0,j=1,\\
\mu_1(t)h\dif t + o(\dif t), & i=-1,j=0,\\
\mu_2(t)k\dif t + o(\dif t), & i=0,j=-1,\\
\eta_1(t)h \dif t + o(\dif t), & i=-1,j=1,\\
\eta_{2}(t)k\dif t + o(\dif t), & i=1,j=-1,\\
1-\big(\lambda_1(t)h+\lambda_2(t)k+\mu_1(t)h+\mu_2(t)k + \eta_{1}(t)h + \eta_{2}(t)k \big) \dif t + o(\dif t), & i,j=0, \\
o(\dif t), & \text{otherwise},
\end{array}
\end{cases} \nonumber
\end{align}
where $\lambda_1,\lambda_2,\mu_1,\mu_2,\eta_1,\eta_2$ are integrable functions.
\\

From the definition \eqref{definizioneInfinitesimaleProcessoNascitaMorteInterazione}, by means of classical arguments we derive the state equation. By denoting with $p_{h,k}(t) = P\{N_1(t)= h, N_2(t) = k\}$ (i.e. $p_{h,k}(t) = 0$ if $h,k\not\in\mathbb{N}_0^2$), $t\ge0$ (we omit the explicit dependence of the time for the rate functions)
\begin{align}
\frac{\dif p_{h,k}(t)}{\dif t} &= -\Bigl((\lambda_1+\mu_1+\eta_1) h+(\lambda_2+\mu_2+\eta_2)k\Big)p_{h,k}(t) \nonumber\\
&\ \ \ + \lambda_1 (h-1)p_{h-1,k}(t) + \lambda_2 (k-1)p_{h, k-1}(t)\nonumber\\
&\ \ \ + \mu_1 (h+1)p_{h+1,k}(t) + \mu_2 (k+1)p_{h, k+1}(t) \nonumber\\
&\ \ \ + \eta_1 (h+1)p_{h+1,k-1}(t) + \eta_2(k+1)p_{h-1, k+1}(t),\ \ \  h,k\ge 0, \label{equazioneStatoProcessoNascitaMorteInterazione}
\end{align}
with initial conditions $p_{h,k}(0) = 1$ if $(h,k) = (n_1,n_2)$ and $0$ otherwise.

Note that for a non-linear process, formula \eqref{definizioneInfinitesimaleProcessoNascitaMorteInterazione} and equation \eqref{equazioneStatoProcessoNascitaMorteInterazione} suitably modify. Furthermore, one may imagine several types of interaction between the two populations, also with migration rates depending on the number of individuals in both groups (this kind of hypotheses may raise several mathematical complications).
\\

From the state equation \eqref{equazioneStatoProcessoNascitaMorteInterazione} we derive the differential equation for the joint probability generating function $G(t,u,v) = \sum_{h=0}^{n_1}\sum_{k=0}^{n_2} u^h v^k p_{h,k}(t)$, with $u,v$ in the neighborhood of $0$,
\begin{align}\label{equazionDifferenzialeGeneratriceProbabilitaNascitaMorteMigrazione}
&\frac{\partial G}{\partial t} = \Bigl(\lambda_1 u^2 - (\lambda_1 + \mu_1+\eta_1) u+\mu_1+\eta_1v\Big)\frac{\partial G}{\partial u} + \Bigl(\lambda_2 v^2 - (\lambda_2 + \mu_2+\eta_2) v+\mu_2+\eta_2u\Big)\frac{\partial G}{\partial v},
\end{align}
with initial condition $G(0, u,v) = u^{n_1}v^{n_2}$.
\\

Hereafter we restrict ourselves to the homogeneous case, where the rate functions are constant.

\begin{proposition}\label{proposizioneValoreAttesoNascitaMorteInterazione}
Let $(N_1,N_2)$ be the (homogeneous) birth-death-migration vector process defined in \eqref{definizioneInfinitesimaleProcessoNascitaMorteInterazione}. Then, for $t\ge0$,
\begin{align}
&\mathbb{E}N_1(t)=\frac{n_1e^{-Lt/2}}{2}\bigl(e^{Rt/2} + e^{-Rt/2}\bigr) +  \frac{e^{-Lt/2}}{R}\Bigl(n_2\eta_2 - \frac{n_1M}{2}\Bigr)\bigl(e^{Rt/2} - e^{-Rt/2}\bigr), \label{momentoPrimoComponenteNascitaMorteMigrazione}
\end{align}
where $L = -\lambda_1 +\mu_1+\eta_1-\lambda_2+\mu_2+\eta_2, \ M = (-\lambda_1 + \mu_1+\eta_1) -(-\lambda_2+\mu_2+\eta_2)$ and $R =\sqrt{M^2 + 4\eta_1\eta_2}$. Similarly for $N_2(t)$ (it sufficies to switch the subscripts, also in $M$).
\end{proposition}

\begin{proof}
We derive \eqref{equazionDifferenzialeGeneratriceProbabilitaNascitaMorteMigrazione} with respect to $u$ and set $u=v=1$. Then we repeat with respect to $v$. With some calculation, this operation leads to the following linear first-order differential system
\begin{equation}\label{sistemaDifferenzialeValoreAttesoNascitaMorteInterazione}
\begin{cases}
	\frac{\dif  }{\dif  t}\mathbb{E} N_1(t) = (\lambda_1-\mu_1-\eta_1)\mathbb{E}N_1(t) + \eta_2\mathbb{E}N_2(t),\\
	\frac{\dif  }{\dif  t} \mathbb{E} N_2(t) = (\lambda_2-\mu_2-\eta_2)\mathbb{E}N_2(t) + \eta_1\mathbb{E}N_1(t),
\end{cases}
\end{equation}
with the initial condition $\mathbb{E}N_1(0) = n_1$ and $\mathbb{E}N_2(0) = n_2$. By solving \eqref{sistemaDifferenzialeValoreAttesoNascitaMorteInterazione} with classical methods we obtain the stated result.
\end{proof}

The interested reader can check that if $\eta_1 = \eta_2 =0 $ the moment \eqref{momentoPrimoComponenteNascitaMorteMigrazione} is equal to $e^{(\lambda_1-\mu_1)t}$ since $N_1$ reduces to a classical birth-death process.

\begin{remark}[Second-order moments]
Note that by deriving \eqref{equazionDifferenzialeGeneratriceProbabilitaNascitaMorteMigrazione} with respect to both $u$ and $v$, by deriving it twice with respect to $u$ and by deriving twice with respect to $v$ and always setting $u=v=1$, we obtain a non-homogeneous linear differential system concerning the second order moments $\sigma(t) = \mathbb{E}N_1(t)N_2(t), \sigma_1(t) = \mathbb{E}N_1(t)(N_1(t)-1)$ and $\sigma_2(t) = \mathbb{E}N_2(t)(N_2(t)-1),\  t\ge0,$
 reading
\begin{equation}\label{sistemaDifferenzialeMomentiSecondiNascitaMorteInterazione}
	\begin{cases}
		\sigma'(t) = (\lambda_1-\mu_1-\eta_1+\lambda_2-\mu_2-\eta_2)\sigma(t)+\eta_1\sigma_1(t)+\eta_2\sigma_2(t),\\
		\sigma_1'(t)= 2(\lambda_1-\mu_1-\eta_1)\sigma_1(t) +2\eta_2\sigma(t) +2\lambda_1 \mathbb{E}N_1(t),\\
		\sigma_2'(t)= 2(\lambda_2-\mu_2-\eta_2)\sigma_2(t) +2\eta_1\sigma(t) +2\lambda_2 \mathbb{E}N_2(t).
	\end{cases}
\end{equation}
By means of some calculation, system \eqref{sistemaDifferenzialeMomentiSecondiNascitaMorteInterazione} yields a non-homogeneous third-order ordinary differential equations for $\sigma_1$ which we have not been able to solve.

On the other hand, by assuming that $\lambda_1-\mu_1 = \lambda_2-\mu_2 = \alpha>0$ and $\eta_1=\eta_2 = \eta>0$, we can solve system \eqref{sistemaDifferenzialeMomentiSecondiNascitaMorteInterazione} since it reduces to
\begin{equation}\label{sistemaDifferenzialeMomentiSecondiNascitaMorteInterazioneRidotto}
	\begin{cases}
		\sigma'(t) = 2(\alpha-\eta)\sigma(t)+\eta\big(\sigma_1(t)+\sigma_2(t)\big), \\
		 \sigma_1'(t)= 2(\alpha-\eta)\sigma_1(t) +2\eta\sigma(t) +2\lambda_1 \mathbb{E}N_1(t),\\
		 \sigma_2'(t)= 2(\alpha-\eta_2)\sigma_2(t) +2\eta\sigma(t) +2\lambda_2 \mathbb{E}N_2(t).
	\end{cases}
\end{equation}
We first derive a non-homogeneous second-order differential equation for the crossed moment $\sigma$ by means of classical computation (that is by deriving the first equation of \eqref{sistemaDifferenzialeMomentiSecondiNascitaMorteInterazioneRidotto}, then using the second and third equations to replace the derivatives of $\sigma_1$ and $\sigma_2$, and finally using again the first equation to express $\sigma_1+\sigma_2$ in terms of $\sigma$),
\begin{equation}\label{equazioneDifferenzialeSecondoOrdineMomentoSecondoRidotto}
 \sigma''(t) + 4(\eta-\alpha) \sigma'(t)+ 4\alpha (\alpha-2\eta)\sigma = 2\eta\big(\lambda_1\mathbb{E}N_1(t) + \lambda_2\mathbb{E}N_2(t)\big).
\end{equation}
Note that under the given hypotheses, $L=-2(\alpha-\eta), M=0,  R=2\eta$, thus the moments in Proposition \ref{proposizioneValoreAttesoNascitaMorteInterazione} simplify and we have that
\begin{align}
	2\big(\lambda_1&\mathbb{E}N_1(t) + \lambda_2\mathbb{E}N_2(t) \big) \nonumber\\
	&= 2\lambda_1 \Bigl(\frac{n_1+n_2}{2} e^{\alpha t}+\frac{n_1-n_2}{2}e^{(\alpha-2\eta)t}\Bigr) + 2\lambda_2 \Bigl(\frac{n_1+n_2}{2} e^{\alpha t}+\frac{n_2-n_1}{2}e^{(\alpha-2\eta)t}\Bigr)\nonumber\\
	&= (\lambda_1+\lambda_2)(n_1+n_2)e^{\alpha t} +(\lambda_1-\lambda_2)(n_1-n_2)e^{(\alpha-2\eta)t}\nonumber\\
	& = ae^{\alpha t} +be^{(\alpha-2\eta)t}, \label{combinazioneValoriAttesiPrimiNascitaMorteInterazioneRidotto}
\end{align}
where $a = (\lambda_1+\lambda_2)(n_1+n_2) $ and $b = (\lambda_1-\lambda_2)(n_1-n_2)$.
Inspired by the form of \eqref{combinazioneValoriAttesiPrimiNascitaMorteInterazioneRidotto}, by means of the method of the undetermined coefficients, one obtains that a particular solution to equation \eqref{equazioneDifferenzialeSecondoOrdineMomentoSecondoRidotto} reads
$$ f(t) =  \frac{ a \eta}{\alpha(\alpha-4\eta)}e^{\alpha t} +\frac{b \eta}{\alpha^2-4\eta^2}e^{(\alpha-2\eta)t}.$$
Then, by solving the homogeneous differential equation associated to \eqref{equazioneDifferenzialeSecondoOrdineMomentoSecondoRidotto} we obtain that the general solution to \ref{equazioneDifferenzialeSecondoOrdineMomentoSecondoRidotto} reads
\begin{equation}
	\sigma(t) = Ae^{2\alpha t} + Be^{2(\alpha-2\eta)t} + f(t).
\end{equation}
Finally, the initial conditions $\sigma(0) = n_1n_2$ and $\sigma'(0) = 2n_1n_2(\alpha-\eta)+\eta\big(n_1(n_1-1)+n_2(n_2-1)\big)$ (obtained by means of the first equation of \eqref{sistemaDifferenzialeMomentiSecondiNascitaMorteInterazioneRidotto} and keeping in mind that $\sigma_1(0) = n_1(n_1-1)$ and $\sigma_2(0) = n_2(n_2-1)$), we obtain the exact coefficients
\begin{align}
	&A=\frac{n_1n_2}{2}-\frac{n_1(n_1-1)+n_2(n_2-1)}{4}+\frac{a}{4\alpha}+\frac{b}{4(\alpha+2\eta)},\nonumber\\
	&B=\frac{n_1n_2}{2}+\frac{n_1(n_1-1)+n_2(n_2-1)}{4}-\frac{a}{4(\alpha-4\eta)}-\frac{b}{4(\alpha-2\eta)}.\label{coefficientiSoluzioneMomentoIncrociatoNascitaMorteInterazione}
\end{align}

Now, the second equation of \eqref{sistemaDifferenzialeMomentiSecondiNascitaMorteInterazioneRidotto} is a non-homogeneous first-order differential equation and its solution follows by means of some calculation (by also keeping in mund the initial condition $\sigma_1(0) = n_1(n_1-1)$),
\begin{align}
	\sigma_1(t) &= -\frac{1}{\alpha-2\eta}\Bigl(\frac{2a\eta^2}{\alpha(\alpha-4\eta)}+\lambda_1(n_1+n_2)\Bigr) e^{\alpha t}+ Ae^{2\alpha t}\nonumber\\
	&\ \ \ -\frac{1}{\alpha}\Bigl(\frac{2b\eta^2}{\alpha^2-4\eta^2}+\lambda_1(n_1-n_2)\Bigr) e^{(\alpha-2\eta) t}- Be^{2(\alpha-2\eta) t}\nonumber\\
	&\ \ \ \biggl[-A + B + \frac{2\eta^2}{\alpha(\alpha-2\eta)}\Big(\frac{a}{\alpha-4\eta}+\frac{b}{\alpha+2\eta}+2\lambda_1\big[n_1(\alpha-\eta) + n_2\eta\bigr] \Big) + n_1(n_1-1)\bigg]e^{2(\alpha-\eta)t}, \label{momentoSecondoMarginaleProcessoNascitaMorteInterazioneRidotto} 
\end{align}
where $a,b,A,B$ are the coefficients appearing in \eqref{combinazioneValoriAttesiPrimiNascitaMorteInterazioneRidotto} and in \eqref{coefficientiSoluzioneMomentoIncrociatoNascitaMorteInterazione}.
\\
Function $\sigma_2$ follows equivalently (it sufficies to switch the subscripts in \eqref{momentoSecondoMarginaleProcessoNascitaMorteInterazioneRidotto}). \hfill$\diamond$
\end{remark}

We point out that the differential systems \eqref{sistemaDifferenzialeValoreAttesoNascitaMorteInterazione} and \eqref{sistemaDifferenzialeMomentiSecondiNascitaMorteInterazione} can be equivalently derived from \eqref{equazionDifferenzialeGeneratriceProbabilitaNascitaMorteMigrazione} by using the joint cumulant generating function of $(N_1(t),N_2(t))$ as shown in Section 4 of \cite{K1948}.
\\

In the next section we provide a detailed study of the death-migration vector process.

\subsection{Death-migration process}

We now assume the process defined in \eqref{definizioneInfinitesimaleProcessoNascitaMorteInterazione} with constant rate functions and $\lambda_1 = \lambda_2 = 0$. Therefore $(N_1,N_2)$ is a linear death-migration process (or death process with interaction).

We begin by deriving the explicit expression of the probability generating function of the interacting death vector process. From the differential equation \eqref{equazionDifferenzialeGeneratriceProbabilitaNascitaMorteMigrazione} we obtain, with $\|(u,v)\| \le1$ and $t\ge0$,
\begin{equation}\label{equazionDifferenzialeGeneratriceProbabilitaMorteMigrazione}
\frac{\partial G}{\partial t} = \Bigl(\mu_1(1-u) + \eta_1(v-u)\Big)\frac{\partial G}{\partial u} + \Bigl(\mu_2(1-v) + \eta_2(u-v)\Big)\frac{\partial G}{\partial v},
\end{equation}
with initial condition $G(0, u,v) = u^{n_1}v^{n_2}$.

By using the change of variable $w = 1-u, z = 1-v$ we obtain
\begin{equation}\label{equazioneDifferenzialeGeneratriceMorteMigrazioneICambioVariabile}
\frac{\partial G}{\partial t} = -\Bigl(\mu_1w + \eta_1(w-z)\Big)\frac{\partial G}{\partial w} - \Bigl(\mu_2 z + \eta_2(z-w)\Big)\frac{\partial G}{\partial z}.
\end{equation}
We now consider a second change of variable $x = aw+bv, y = cw+dz$ with $a,b,c,d$ such that $$\eta_1 a^2 + (\mu_1-\mu_2+\eta_1-\eta_2)ab - \eta_2 b^2 = \eta_1 c^2 + (\mu_1+\eta_1-\mu_2-\eta_2)cd - \eta_2 d^2 = 0.$$ Thus, we consider $b,d\not =0$ and
\begin{align}\label{coefficientiSoluzioneGeneratriceProbabilitaMorteInterazione}
&a = b\,\frac{-(\mu_1-\mu_2+\eta_1-\eta_2) + \sqrt{ (\mu_1-\mu_2+\eta_1-\eta_2)^2+4\eta_1\eta_2}}{2\eta_1},\nonumber\\
&c = - d\,\frac{(\mu_1-\mu_2+\eta_1-\eta_2) + \sqrt{ (\mu_1-\mu_2+\eta_1-\eta_2)^2+4\eta_1\eta_2}}{2\eta_1}. 
\end{align}
This provides the following determinant for the matrix of the coefficients,
\begin{equation}\label{determinanteCambioVariabileDimostrazioneProcessoMorteInterazione}
ad-bc = bd\frac{\sqrt{ (\mu_1-\mu_2+\eta_1-\eta_2)^2+4\eta_1\eta_2}}{\eta_1} \not = 0, 
\end{equation}
and reduces equation \eqref{equazioneDifferenzialeGeneratriceMorteMigrazioneICambioVariabile} into
\begin{align}
\frac{\partial G}{\partial t} &= -\frac{x}{ad-bc}\Big[ad\mu_1 + a(c+d)\eta_1- bc\mu_2 - b(c+d)\eta_2\Big]\frac{\partial G}{\partial x}\nonumber\\
& \ \ \  -  \frac{y}{ad-bc}\Big[-bc\mu_1 - c(a+b)\eta_1 + ad\mu_2 + d(a+b)\eta_2\Big]\frac{\partial G}{\partial y}\nonumber\\
& = -x\Big[\frac{\mu_1+\mu_2+\eta_1+\eta_2}{2}-\frac{\sqrt{(\mu_1-\mu_2+\eta_1-\eta_2)^2+4\eta_1\eta_2}}{2}\Big]\frac{\partial G}{\partial x}\nonumber\\
&\ \ \  -y\Big[\frac{\mu_1+\mu_2+\eta_1+\eta_2}{2}+\frac{\sqrt{(\mu_1-\mu_2+\eta_1-\eta_2)^2+4\eta_1\eta_2}}{2}\Big]\frac{\partial G}{\partial y},\label{equazioneDifferenzialeGeneratriceMorteMigrazioneIICambioVariabile}
\end{align}
where in the last step we replaced $a$ and $c$ with the quantities in \eqref{coefficientiSoluzioneGeneratriceProbabilitaMorteInterazione}, used \eqref{determinanteCambioVariabileDimostrazioneProcessoMorteInterazione} and performed some calculation.
\\
We denote by $L =\mu_1+\mu_2+\eta_1+\eta_2$ and
\begin{equation}\label{notazioneGeneratriceProbabilitaMorteMigrazione}
M = \mu_1-\mu_2+\eta_1 - \eta_2,\ \ \ R =\sqrt{M^2+4\eta_1\eta_2}= \sqrt{(\mu_1-\mu_2+\eta_1-\eta_2)^2+4\eta_1\eta_2}.
\end{equation}
Now, it is easy to check that the function $f(xe^{-(L-R)t/2}, ye^{-(L+R)t/2})$ satisfies equation \eqref{equazioneDifferenzialeGeneratriceMorteMigrazioneIICambioVariabile}. Thus, retuning to the original variables $u,v$ we obtain that
\begin{equation}\label{generatriceProbabilitaMorteMigrazioneFormaConF}
 G(t, u,v) = f\Big(e^{-\frac{L-R}{2}t}\big[a(1-u) + b(1-v)\big], e^{-\frac{L+R}{2}t}\big[c(1-u) + d(1-v)\big]\Big).
\end{equation}
Finally, by considering the initial condition $G(0,u,v) = u^{n_1}v^{n_2}$ we obtain the explicit form for $f$, that is 
\begin{equation}\label{generatriceProbabilitaMorteMigrazioneFormaEsplicitaF}
 f(w,z) = \Big(1-\frac{dw-bz}{ad-bc}\Big)^{n_1}\Big(1-\frac{az-cw}{ad-bc}\Big)^{n_2}.
\end{equation}
Hence, by keeping in mind the coefficients in \eqref{coefficientiSoluzioneGeneratriceProbabilitaMorteInterazione} and the notation \eqref{notazioneGeneratriceProbabilitaMorteMigrazione}, the probability generating function follows by putting together \eqref{generatriceProbabilitaMorteMigrazioneFormaConF} and \eqref{generatriceProbabilitaMorteMigrazioneFormaEsplicitaF},
\begin{align}
&G(t,u,v)\nonumber\\
&= \Biggl(1-\frac{e^{-\frac{\mu_1+\mu_2+\eta_1+\eta_2}{2}t} }{2R}\bigg[  (1-u)\Bigl((R-M)e^{\frac{R}{2}t} + (R+M)e^{-\frac{R}{2}t}\Bigr) +  (1-v)2\eta_1 \Bigl(e^{\frac{R}{2}t} - e^{-\frac{R}{2}t}\Bigr)\bigg]\Biggr)^{n_1}\nonumber\\
& \ \ \ \times \Biggl(1-\frac{e^{-\frac{\mu_1+\mu_2+\eta_1+\eta_2}{2}t} }{2R}\bigg[  (1-u)2\eta_2 \Bigl(e^{\frac{R}{2}t} - e^{-\frac{R}{2}t}\Bigr) + (1-v)\Bigl((R+M)e^{\frac{R}{2}t} + (R-M)e^{-\frac{R}{2}t}\Bigr)  \bigg]\Biggr)^{n_2} \label{generatriceProbabilitaMorteMigrazione} \\
& = \Biggl(1-e^{-\frac{\mu_1+\mu_2+\eta_1+\eta_2}{2}t}\bigg[  (1-u)\Bigl(\cosh(Rt/2) - \frac{M}{R}\sinh(Rt/2)\Bigr) + (1-v)\frac{2\eta_1}{R}\sinh(Rt/2)\bigg]\Biggr)^{n_1} \nonumber\\
& \ \ \ \times \Biggl(1-e^{-\frac{\mu_1+\mu_2+\eta_1+\eta_2}{2}t}\bigg[  (1-u)\frac{2\eta_2}{R}\sinh(Rt/2)+(1-v)\Bigl(\cosh(Rt/2) + \frac{M}{R}\sinh(Rt/2)\bigg)\bigg]\Biggr)^{n_2}. \label{generatriceProbabilitaMorteMigrazioneCosSinIperbolico} 
\end{align}

The above result permits us to derive the following interesting decomposition for the death-migration vector process $(N_1,N_2)$.

\begin{theorem}\label{teoremaMorteMigrazioneSommaMultinomiali}
Let $(N_1,N_2)$ be the death-migration vector process described above. Then, 
\begin{equation}
N_1(t) = X_1(t)+Y_1(t)  \ \ \text{ and } \ \ N_2(t) = X_2(t)+Y_2(t),
\end{equation}
where $X(t) = \bigl(X_1(t), X_2(t)\big)\sim Multinomial\big(n_1, A_1(t), B_1(t)\big)$ and $Y(t) = \bigl(Y_1(t), Y_2(t)\big)\sim Multinomial\big(n_2, A_2(t), B_2(t)\big)$ are independent random variables, with $M,R$ defined in \eqref{notazioneGeneratriceProbabilitaMorteMigrazione} and
\begin{align}
&A_1(t) = \frac{e^{-\frac{\mu_1+\mu_2+\eta_1+\eta_2}{2}t} }{2R}\Bigl((R-M)e^{\frac{R}{2}t} + (R+M)e^{-\frac{R}{2}t}\Bigr) ,\nonumber\\
&B_1(t) = \frac{e^{-\frac{\mu_1+\mu_2+\eta_1+\eta_2}{2}t} }{R}\eta_1 \Bigl(e^{\frac{R}{2}t} - e^{-\frac{R}{2}t}\Bigr),\nonumber\\
&A_2(t) = \frac{e^{-\frac{\mu_1+\mu_2+\eta_1+\eta_2}{2}t} }{R}\eta_2 \Bigl(e^{\frac{R}{2}t} - e^{-\frac{R}{2}t}\Bigr),\nonumber\\
&B_2(t) = \frac{e^{-\frac{\mu_1+\mu_2+\eta_1+\eta_2}{2}t} }{2R}\Bigl((R+M)e^{\frac{R}{2}t} + (R-M)e^{-\frac{R}{2}t}\Bigr).\label{coefficientiMultinomialiProcessoMorteInterazioni}
\end{align}
\end{theorem}

\begin{proof}
First, one can check that the coefficients in \eqref{coefficientiMultinomialiProcessoMorteInterazioni} are in $(0,1)$ for all $t\ge0$. Then, in light of notation \eqref{coefficientiMultinomialiProcessoMorteInterazioni}, the probability generating function \eqref{generatriceProbabilitaMorteMigrazione} reads 
\begin{equation}\label{generatriceProbabilitaProcessoMorteInterazioneNotazione}
G(t,u,v) = \Bigl(1- A_1(t)(1-u) - B_1(t)(1-v)\Bigr)^{n_1} \Bigl(1- A_2(t)(1-u) - B_2(t)(1-v)\Bigr)^{n_2} 
\end{equation}
which yields the claimed result, reminding that the probability generating function of a multinomial random variable $X = (X_1,X_2)\sim Multinomial (n,p,q)$ is 
$$\mathbb{E}u^{X_1}v^{X_2} = (1-p(1-u)-q(1-v)\big)^n,$$
with $|u|,|v|\le1$, $p,q\in(0,1)$ and $p+q<1$.
\end{proof}

Formula \eqref{generatriceProbabilitaProcessoMorteInterazioneNotazione} permits us to derive the moments of the vector $(N_1, N_2)$. The interestted reader can check that the first moments are coherent with those appearing in Proposition \ref{proposizioneValoreAttesoNascitaMorteInterazione}. For instance the covariance of the two components is, for $t\ge0$,
$$\text{Cov}\bigl(N_1(t),N_2(t)\bigr) = -n_1A_1(t)B_1(t)- n_2A_2(t)B_2(t) $$
which is negative because of the migration behavior (algebraically, $A_1,B_1,A_2,B_2$ are non-negative functions of $t$, see \eqref{coefficientiMultinomialiProcessoMorteInterazioni}).
\\

Theorem \ref{teoremaMorteMigrazioneSommaMultinomiali} permits us to study combinations of the components $N_1,N_2$. It is interesting to notice that, for fixed $t\ge0$, the sum of the components of $(N_1, N_2)$ is the sum of independent binomial random variables, $N_1(t) + N_2(t) = U(t)+V(t)$, with $U(t)\sim Binomial\big(n_1,A_1(t)+B_1(t)\big)$ and $V(t)\sim Binomial\big(n_2,A_2(t)+B_2(t)\big)$ with the parameters given in \eqref{coefficientiMultinomialiProcessoMorteInterazioni}.

\begin{remark}[Extinction]
From the probability generating function we easily derive the extinction probability of the vector process. In particular, notation \eqref{coefficientiMultinomialiProcessoMorteInterazioni}, formulas \eqref{generatriceProbabilitaProcessoMorteInterazioneNotazione} and \eqref{generatriceProbabilitaMorteMigrazioneCosSinIperbolico} provide, for $t\ge0$,
\begin{align}
P\{N_1(t) = 0, N_2(t) = 0\} &= \big(1- A_1(t)- B_1(t)\big)^{n_1} \big(1- A_2(t)- B_2(t)\big)^{n_2}\label{probabilitaEstinzione}\\
& =  \Biggl(1-e^{-\frac{\mu_1+\mu_2+\eta_1+\eta_2}{2}t}\Bigl[\cosh(Rt/2) + \frac{\mu_2+\eta_2-\mu_1+\eta_1}{R}\sinh(Rt/2)\Bigr]\Biggr)^{n_1} \nonumber\\
& \ \ \ \times \Biggl(1-e^{-\frac{\mu_1+\mu_2+\eta_1+\eta_2}{2}t}\Bigl[\cosh(Rt/2) + \frac{\mu_1+\eta_1-\mu_2+\eta_2}{R}\sinh(Rt/2)\Bigr]\Biggr)^{n_2}.\nonumber
\end{align}
Note that as $t\longrightarrow 0$ the extinction probability goes to $1$.

We now consider $\mu_1 = \mu_2 =\mu >0$. In this case, $A_1(t)+B_1(t) = A_2(t) + B_2(t) = e^{-\mu t}\ \forall\ t$ and probability \eqref{probabilitaEstinzione} does not depend neither on $\eta_1,\eta_2$ nor on the distribution of the initial $n_1+n_2$ individuals among the two groups, reading $P\{N_1(t) = 0, N_2(t) = 0\} =(1-e^{-\mu t})^{n_1+n_2}$. This yields the expected extinction time of the vector $(N_1, N_2)$, $T_0 = \inf\{t>0\,:\, N_1(t) = 0, N_2(t) = 0\}$, for $t\ge0$,
\begin{align}
\mathbb{E} T_0 &= \int_0^\infty t P\{T_0\in \dif t\} = \int_0^\infty t \frac{\dif }{\dif t} P\{N_1(t) = 0, N_2(t) = 0\} \dif t \nonumber\\
& =  \int_0^\infty t (n_1+n_2)\mu e^{-\mu t}\bigl( 1-e^{-\mu t}\big)^{n_1+n_2-1}\dif t = \frac{n_1+n_2}{\mu}\sum_{k=0}^{n_1+n_2-1} \binom{n_1+n_2-1}{k}\frac{(-1)^k}{(1+k)^2} \nonumber\\
&= \frac{1}{\mu}\sum_{k=1}^{n_1+n_2} \binom{n_1+n_2}{k}\frac{(-1)^{k+1}}{k} = \sum_{k=1}^{n_1+n_2} \frac{1}{\mu k}, \label{attesaEstinzioneMortiUguali}
\end{align}
where in the last step we used the relationship
\begin{equation}
\sum_{k=1}^{n_1+n_2} \binom{n_1+n_2}{k}\frac{(-1)^{k}}{k} = - \sum_{k=1}^{n_1+n_2} \frac{1}{ k} = -H_{n_1+n_2},\label{relazioneFormaAlternativaNumeroArmonico}
\end{equation}
with $H_{n_1+n_2}$ being the $(n_1+n_2)$-th harmonic number (see Appendix \ref{appendiceDimostrazioneRelazioneFormaAlternativaNumeroArmonico} for the proof of \eqref{relazioneFormaAlternativaNumeroArmonico}).
\\

Note that (also in the case of $\mu_1\not =\mu_2$) by keeping in mind that the sum $N_1+N_2$ is a non-increasing process, the following generating function holds true (we refer to Appendix A of \cite{CO2025} for a similar proof),
\begin{equation}\label{generatriceTempiPrimoPassaggioPopolazioneTotaleMorte}
\sum_{k=0}^{n_1+n_2} u^k P\{T_k >t\}= \frac{1-\mathbb{E} u^{N_1(t) +N_2(t)}}{1-u} = \frac{1-G(t,u,u)}{1-u},\ \ \ t\ge0,\ |u|\le 1,
\end{equation}
where $T_k = \inf\{t>0\,:\, N_1(t) +N_2(t) = k\}$ is the first time that the total population is equal to $0\le k\le n_1+n_2$ and $G$ is the probability generating function \eqref{generatriceProbabilitaMorteMigrazione}. Note that if $k=0$, $T_0$ is the same one appearing in \eqref{attesaEstinzioneMortiUguali} and that formula \eqref{generatriceTempiPrimoPassaggioPopolazioneTotaleMorte}, with $u=1$, coincides with the complement to one of probability \eqref{probabilitaEstinzione}.
\\

Finally, we observe that, in the case of death rates equal to $\mu>0$, with $t\ge 0$,
$P\{T_0\le t\} = (1-e^{-\mu t})^{n_1+n_2} = P\{\max\{T_1, T_2\}\le t\}$
where $T_1$ and $T_2$ are the extinction times of two independent linear death processes of rates $\mu>0$ whose sum of initial units is $n_1+n_2$.\hfill $\diamond$
\end{remark}

\begin{remark}
In the case of symmetric death rates $\mu_1 = \mu_2 = \mu>0$, the probability generating function \eqref{generatriceProbabilitaMorteMigrazione} turns into
\begin{align}
G(t,u,v) &= \Bigg(1-\frac{e^{-\mu t}}{\eta_1+\eta_2}\Big[(1-u)\bigl(\eta_2+\eta_1e^{-(\eta_1+\eta_2) t}\big)+(1-v)\eta_1\bigl(1-e^{-(\eta_1+\eta_2) t}\big)\Big]\Bigg)^{n_1}\nonumber\\
& \ \ \ \times\Bigg(1-\frac{e^{-\mu t}}{\eta_1+\eta_2}\Big[(1-u)\eta_2\bigl(1-e^{-(\eta_1+\eta_2) t}\big)+(1-v)\bigl(\eta_1+\eta_2e^{-(\eta_1+\eta_2) t}\big)\Big]\Bigg)^{n_1}.\label{generatriceProbabilitaMorteMigrazioneMorteSimmetrica}
\end{align}

The interested reader can also evaluate the case when $\mu_1+\eta_1 = \mu_2+\eta_2$, meaning that the rates of the groups are balanced. Indeed, formula \eqref{generatriceProbabilitaMorteMigrazione} simplifies a lot since $M=0, R = 2\sqrt{\eta_1\eta_2}$.
\\

Finally, if we assume a complete symmetry in the rates $\mu_1 = \mu_2, \eta_1 = \eta_2$, we can solve the differential equation \eqref{equazionDifferenzialeGeneratriceProbabilitaMorteMigrazione}  with time-depending rates, i.e. the non-homogeneous version of the birth-migration process. Its probability generating function reads
\begin{align}
G(t,u,v) &= \bigg[1-\frac{e^{-\int_0^t\mu(s) \dif s}}{2}\Big((1-u)\big(1+e^{-2\int_0^t\eta(s) \dif s t}\big)+(1-v)\big(1-e^{-2\int_0^t\eta(s) \dif s}\big)\Big)\bigg]^{n_1}\nonumber\\
& \ \ \ \times\bigg[1-\frac{e^{-\int_0^t\mu(s) \dif s}}{2}\Big((1-u)\big(1-e^{-2\int_0^t\eta(s) \dif s }\big)+(1-v)\big(1+e^{-2\int_0^t\eta(s) \dif s}\big)\Big)\bigg]^{n_2},\label{generatriceProbabilitaMorteMigrazioneCasoSimmetrico}
\end{align}
where $\mu,\eta$ are integrable rate functions. We point out that Theorem \ref{teoremaMorteMigrazioneSommaMultinomiali} suitably holds in this non-homogeneous case.
\hfill$\diamond$
\end{remark}

\begin{remark}[Pure migration process]
Let $\mu_1 = \mu_2 = 0$, then no units of the couple $(N_1,N_2)$ can die. Thus, the death-migration process described above simplifies and it can be written as $(N_1, N_2 = n_1+n_2-N_1)$ where a unit can leave $N_1$ with a rate proportional to $N_1$ and a unit can arrive in $N_1$ with a rate proportional to $n_1+n_2-N_1$. The coefficients in \eqref{notazioneGeneratriceProbabilitaMorteMigrazione} reduce to $M = \eta_1-\eta_2$ and $R = \eta_1+\eta_2$ and those in \eqref{coefficientiMultinomialiProcessoMorteInterazioni} turn into
\begin{align} \label{coefficientiMultinomialiProcessoMigrazionePura}
&\bar A_1(t) =\frac{\eta_2 + \eta_1e^{-(\eta_1+\eta_2)t}}{\eta_1+\eta_2},\ \ \ \bar B_2(t) = 1- \bar A_1(t),\nonumber\\
&\bar A_2(t) = \eta_2\frac{1 - e^{-(\eta_1+\eta_2)t}}{\eta_1+\eta_2},\ \ \ \  \bar B_2(t) = 1- \bar A_2(t).
\end{align}
The migration vector process can be described by one component, say $N_1$, since the other one is $N_2 = n_1+n_2-N_1$. We use the results concerning the death-migration vector process to derive explicit results on the migration process $N_1$. Indeed, in light of the probability generating function $G$ in \eqref{generatriceProbabilitaMorteMigrazione} and the new coefficients in \eqref{coefficientiMultinomialiProcessoMigrazionePura}, we can write (performing some simple calculation), with $t\ge0$ and $|u|\le 1$,
\begin{align} \label{funzioneGeneratriceProbabilitaProcessoMigrazionePura}
\mathbb{E} u^{N_1(t)}&= G(t, u, 1) = \big(1-(1-u)\bar A_1(t)\big)^{n_1}\big(1-(1-u)\bar A_2(t)\big)^{n_2}  \\
& = u^{n_1+n_2} G(t,1,u^{-1})=  \mathbb{E}u^{n_1+n_2-N_2(t)}, \nonumber
\end{align}
also conferming that $N_1(t)=n_1+n_2 -N_2(t)\ \forall\ t$.

Theorem \ref{teoremaMorteMigrazioneSommaMultinomiali} permits us to derive the probability law of the migration process $N_1$, that is, for fixed $t\ge0$, 
\begin{equation}\label{processoMigrazionePuraDecomposizione}
N_1(t)  = X_t+ Y_t, \ \ \ X_t \sim Binomial\big(n_1, \bar A_1(t)\big),\ Y_t\sim Binomial\big(n_2, \bar A_2(t)\big),
\end{equation}
with $X_t,Y_t$ independent random variables. 

Since $N_1$ is an irreducible continuous-time Markov chain on the finite state space $\{0,\dots,n_1+n_2\}$, it admits a stationary distribution. By noticing that $\lim_{t\rightarrow\infty }\bar A_1(t) = \lim_{t\rightarrow\infty }\bar A_2(t) =\eta_2/(\eta_1+\eta_2)$, result \eqref{processoMigrazionePuraDecomposizione} yields
\begin{equation}
N_1(t)\stackrel{d}{ \substack{\xrightarrow{\hspace*{1.7cm}}\\t\longrightarrow \infty }}N\sim Binomial\Big(n_1+n_2, \frac{\eta_2}{\eta_1+\eta_2}\Big).
\end{equation}

Finally, we derive the difference-differential equation governing by the probability law of the migration process $N_1$ and the partial differential equation governing by the probability generating function. The infinitesimal behavior of $N_1$ is, for $t\ge0$ and $k\in\{0,\dots, n_1 + n_2\}$,
\begin{align}\label{definizioneInfinitesimaleProcessoMigrazionePura}
P\{N_1(t+\dif t) = k+i\,|N_1(t) = k\} = 
\begin{cases}\begin{array}{l l}
\eta_1 k\dif t+o(\dif t), & i=-1,\\
\eta_2 (n_1+n_2-k)\dif t +o(\dif t), & i=1,\\
1-\big(\eta_1 k + \eta_2(n_1+n_2-k)\big)\dif t+ o(\dif t), & i=0,\\
o(\dif t),& otherwise.
\end{array}\end{cases}
\end{align}
By keeping in mind that $N_1$ has independent increments, we can derive the governing equation of its probability mass function. Let $p_k(t) = P\{N_1(t) = k\}$ (i.e. $p_k(t) = 0$ if $k\not\in\{0,\dots,n_1+n_2\}$) and $t\ge0$, then
\begin{align}\label{equazioneDifferenzialeLeggeProcessoMigrazionePura}
\frac{\dif p_{k}(t)}{\dif t}&= -\Bigl(\eta_1 k +\eta_2(n_1+n_2-k)\Big) p_{k}(t) \\
& \ \ \ + \eta_1 (k+1)p_{k+1}(t)+\eta_2(n_1+n_2-k)p_{k-1}(t) ,\ \ \ 0\le k\le n_1+n_2,\nonumber 
\end{align}
with initial condition $p_k(0) = 1$ if $k=n_1$ and $p_k(0) = 0$ otherwise.
Note that for a non-linear process, formula \eqref{definizioneInfinitesimaleProcessoMigrazionePura} and equation \eqref{equazioneDifferenzialeLeggeProcessoMigrazionePura} suitably modify. 

From \eqref{equazioneDifferenzialeLeggeProcessoMigrazionePura} some calculation yield the differential equation governing the probability generating function $\bar G(t,u) = \sum_{k=0}^{n_1+n_2} u^k p_k(t)$, with $|u|\le1$,
\begin{equation}\label{equazioneDifferenzialeProbabilitaGeneratriceProcessoMigrazionePura}
\frac{\partial \bar G}{\partial t}= -\eta_2(n_1+n_2)(u-1)\bar G - \big(\eta_2u^2 +u(\eta_1+\eta_2) -\eta_1\big)	\frac{\partial \bar G}{\partial u},
\end{equation}
with initial condition $\bar G(0,u) = u^{n_1}$. Note that the solution to \eqref{equazioneDifferenzialeProbabilitaGeneratriceProcessoMigrazionePura} is given in formula \eqref{funzioneGeneratriceProbabilitaProcessoMigrazionePura},  $\bar G(t,u) =\mathbb{E}u^{N_1(t)} =  G(t,u,1)$. \hfill $\diamond$
\end{remark}

\subsection{Difference of independent birth-death processes}

Inspired by the definition of the Skellam process as the difference of two independent Poisson processes, here we derive some explicit results for the difference of two independent birth-death processes, $N_1$ and $N_2$. We recall that for a (homogeneous) birth-death process $N$, the probability mass at time $t\ge0$ reads
\begin{equation}\label{massaProbabilitaNascitaMorte}
P\{N(t) =k\} = 
\begin{cases}\begin{array}{l l}
\displaystyle e^{-(\lambda-\mu)t}(\lambda-\mu)^2\frac{\big(\lambda(1-e^{-(\lambda-\mu)t})\big)^{k-1}}{\big(\lambda-\mu e^{-(\lambda-\mu)t}\big)^{k+1}},& k\in \mathbb{N},\\
\displaystyle \frac{\mu(1-e^{-(\lambda-\mu)t})}{\lambda -\mu e^{-(\lambda-\mu)t}},& k=0.
\end{array}
\end{cases}
\end{equation}

Now, the distribution of the difference of two independent birth-death processes $N_1,N_2$ with rates $\lambda_1,\mu_1$ and $\lambda_2,\mu_2$ respectively, reads, for $k\in\mathbb{N}, \ t\ge 0,$
\begin{align}
&P\{N_1(t)-N_2(t) = k\} = P\{N_1(t) = k\}P\{N_2(t) = 0\} + \sum_{h=1}^\infty P\{N_1(t) = h+k\} P\{N_2(t) = h\} \nonumber\\
& = P\{N_1(t) = k\}P\{N_2(t) = 0\} + e^{-(\lambda_1-\mu_1)t-(\lambda_2-\mu_2)t}(\lambda_1-\mu_1)^2(\lambda_2-\mu_2)^2 \nonumber\\
&\ \ \ \times\frac{\big(\lambda_1(1-e^{-(\lambda_1-\mu_1)t})\big)^{k-1}}{\big(\lambda_1-\mu_1 e^{-(\lambda_1-\mu_1)t}\big)^{k+1}}\frac{\big(\lambda_2(1-e^{-(\lambda_2-\mu_2)t})\big)^{-1}}{\big(\lambda_2-\mu_2 e^{-(\lambda_2-\mu_2)t}\big)} \sum_{h=1}^\infty \Bigg(\frac{\lambda_1\big(1-e^{-(\lambda_1-\mu_1)t}\big)}{\big(\lambda_1-\mu_1 e^{-(\lambda_1-\mu_1)t}\big)} \frac{\lambda_2\big(1-e^{-(\lambda_2-\mu_2)t}\big)}{\big(\lambda_2-\mu_2 e^{-(\lambda_2-\mu_2)t}\big)}\Bigg)^{h} \nonumber\\
&=  P\{N_1(t) = k\}P\{N_2(t) = 0\} + e^{-(\lambda_1-\mu_1)t-(\lambda_2-\mu_2)t}(\lambda_1-\mu_1)^2(\lambda_2-\mu_2)^2 \nonumber\\
&\ \ \ \times\frac{\big(\lambda_1(1-e^{-(\lambda_1-\mu_1)t})\big)^{k-1}}{\big(\lambda_1-\mu_1 e^{-(\lambda_1-\mu_1)t}\big)^{k+1}}\frac{\big(\lambda_2(1-e^{-(\lambda_2-\mu_2)t})\big)^{-1}}{\big(\lambda_2-\mu_2 e^{-(\lambda_2-\mu_2)t}\big)} \frac{ \frac{\lambda_1\big(1-e^{-(\lambda_1-\mu_1)t}\big)}{\big(\lambda_1-\mu_1 e^{-(\lambda_1-\mu_1)t}\big)} \frac{\lambda_2\big(1-e^{-(\lambda_2-\mu_2)t}\big)}{\big(\lambda_2-\mu_2 e^{-(\lambda_2-\mu_2)t}\big)}}{1-  \frac{\lambda_1\big(1-e^{-(\lambda_1-\mu_1)t}\big)}{\big(\lambda_1-\mu_1 e^{-(\lambda_1-\mu_1)t}\big)} \frac{\lambda_2\big(1-e^{-(\lambda_2-\mu_2)t}\big)}{\big(\lambda_2-\mu_2 e^{-(\lambda_2-\mu_2)t}\big)} }\nonumber \\
&=  P\{N_1(t) = k\} \Biggl( P\{N_2(t) = 0\} + \frac{ \frac{\lambda_1}{\mu_1}P\{N_1(t) = 0\}P\{N_2(t) = 1\}}{1-\frac{\lambda_1}{\mu_1} P\{N_1(t) = 0\}\frac{\lambda_2}{\mu_2} P\{N_2(t) = 0\}} \Biggr),\label{probabilitaDifferenzaNascitaMorteGenerale}
\end{align}
where we used \eqref{massaProbabilitaNascitaMorte} in the last step. The probability for $k<0$ follows in a similar manner (or one can simply switch the subscripts $1$ and $2$ and replace $k$ with $|k|$). By means of the same arguments we also derive the probability that the difference process is equal to $0$, that is 
$$ P\{N_1(t)-N_2(t) = 0\} = P\{N_1(t) = 0\}P\{N_2(t) = 0\} + \frac{P\{N_1(t)=1\}P\{N_2(t)=1\}}{1-\frac{\lambda_1}{\mu_1} P\{N_1(t) = 0\}\frac{\lambda_2}{\mu_2} P\{N_2(t) = 0\}}.$$

Note that in the case of $\lambda_1 = \lambda_2 = \lambda$ and $\mu_1 = \mu_2 = \mu$ and considering $N(t)\stackrel{d}{=}N_1(t)\stackrel{d}{=}N_2(t)\ \forall\ t,$ formula \eqref{probabilitaDifferenzaNascitaMorteGenerale} reduces to 
$$ P\{N_1(t)-N_2(t) = k\} = P\{N(t)= |k|\}\frac{(\lambda+\mu)P\{N(t)= 0\}}{\mu + \lambda P\{N(t)= 0\}},\ \ \ k\in\mathbb{Z}\setminus\{0\}, \ t\ge 0,$$
since
$$ 1+\frac{\frac{\lambda}{\mu}P\{N(t)=1\}}{1-\big(\frac{\lambda}{\mu}P\{N(t)=0\}\big)^2} = \frac{\lambda+\mu}{\mu+\lambda P\{N(t)=0\}},$$ which can be derived by means of \eqref{massaProbabilitaNascitaMorte}.

\section{Time-changed interacting point processes}

In this section we present a result which provides a characterization of the interacting point processes discussed above, time-changed with the inverse of a subordinator with L\'{e}vy symbol a Bernstein function $f$. 

We recall that $f:[0,\infty)\longrightarrow [0,\infty)$ is a Bernstein function if $f\in C^\infty, \ (-1)^n \dif^n f/\dif x^n \le 0\ \forall \ n\ge1$ and it can be expressed as
\begin{equation}
f(x) = a+bx+ \int_0^\infty\Bigl( 1-e^{-xw}\Bigr) \bar \nu_f(\dif w), \ \ \ x\ge0,
\end{equation}
where $a,b\ge0$ and $\bar \nu_f$ is a L\'{e}vy measure, i.e. such that $\int_0^\infty (s\wedge1)\bar\nu_f(\dif s)<\infty$.
Bernstein functions are related to non-decreasing L\'{e}vy processes, also known as subordinators. Indeed, for each Bernstein function $f$ there exists a subordinator $ H_f$ such that $f$ is the L\'{e}vy symbol of $ H_f$, i.e. $\mathbb{E}e^{-\mu H_f(t)} = e^{-tf(\mu)}, \ \mu, t\ge0$. 

The inverse of the subordinator $ H_f$ is the stochastic process $L_f$ defined as
\begin{equation}
L_f(t) = \inf\{x\ge 0\,:\,  H_f(x) \ge t\}, \ \ \ t\ge0,
\end{equation}
which describes the first passage time through level $x$ of the subordinator $ H_f$, therefore, also $L_f$ is non-decreasing (and non-negative). 
\\

Following the line of \cite{BS2024}, we consider the convolution-type derivative with respect to the Bernstein function $f$, appeared in \cite{T2015} (also known as generalized Caputo-Dzerbashyan derivative, see \cite{K2011}), $\mathcal{D}_t^f$, which is defined on the space of absolutely continuous functions as follows
\begin{equation}\label{definizioneDervitaConvoluzionale}
\mathcal{D}_t^f u(t) = b\frac{\dif }{\dif t}u(t) + \int_0^t \frac{\partial }{\partial t}u(t-s)\nu_f(s)\dif s,
\end{equation}
where $\nu_f$ is the tail of the L\'{e}vy measure $\bar \nu_f$.

It is well-known that for $f(x) = x^\alpha$, with $\alpha\in (0,1)$, the L\'{e}vy measure is $\bar\nu_f(s)= \alpha x^{-\alpha-1}/\Gamma(1-\alpha)$, its tail is $\nu_f(x) = x^{-\alpha}/\Gamma(1-\alpha)$ and the operator in \eqref{definizioneDervitaConvoluzionale} reduces to the Caputo-Dzerbashyan derivative. We refer to \cite{BS2024, MT2019, T2015} for further details on the convolution-type derivative $\mathcal{D}_t^f $.
\\

Hereafter we assume the following condition.
\\
\textbf{Condition I.} The L\'{e}vy measure $\bar\nu_f$ associated to the Bernstein function $f$ is such that $\bar\nu(0,\infty) = \infty$ and its tail $\nu_f(s) = a+\bar \nu_f(s,\infty)$ is absolutely continuous.
\\

It was shown in \cite{T2015} that under Condition I, the inverse $L_f$ admits a probability density $l_f(t,x) = P\{L_f(t)\in\dif x\}/\dif x$ with $l_f(0,x) = \delta(x)$, where $\delta$ is the Dirac delta function centered in the origin, and such that its $x$-Laplace transform is an eigenvalue of the operator $\mathcal{D}_f$. Indeed, the function
\begin{equation}\label{trasformataLaplaceLeggeInversoSubordinatore}
\tilde l(t,y) = \mathcal{L}_x\{ l_f\}(t,y) = \int_0^\infty e^{-yx} l(t,x)\dif x = \mathbb{E}e^{-y L_f(t)}, \ \ \ t,y\ge 0,
\end{equation}
where $\mathcal{L}_x$ denotes the $x$-Laplace transform operator, satisfies the equation
\begin{equation*}\label{autofonzioneOperatoreConvoluzionale}
\mathcal{D}_t^f \tilde l(t,y) = -y\tilde l(t,y),\ t,y\ge 0,\ \ \ \tilde l_f(0,y) = 1,\ \forall\ y,
\end{equation*}
also see \cite{K2011} for further details.

In order to prove the next statement we also recall that in \cite{T2015} the author introduced another convolution-type derivative generalizing the Riemann-Liouville fractional derivative, 
\begin{align}
\mathbb{D}_t^f u(t) =  b\frac{\dif }{\dif t}u(t) + \frac{\dif }{\dif t}\int_0^t u(t-s)\nu_f(s)\dif s,
\end{align}
which relates to the convolution-typer derivative \eqref{definizioneDervitaConvoluzionale} by means of (see Proposition 2.7 of \cite{T2015})
\begin{align}\label{relazioneOperatoriConvoluzione}
\mathcal{D}_t^f u(t) =\mathbb{D}_t^f u(t) - \nu_f(t) u(0).
\end{align}
Furthermore, it is useful to recall (see Theorem 4.1 of \cite{T2015}) that the probability density $l_f$ satisfies the problem, with  $x>0$ if $b = 0$ or $0<x<t/b$ if $b>0$,
\begin{align}\label{legameDerivataConvoluzionaleLiouvilleOperatoreSpaziale}
\mathbb{D}_t^f l_f(t,x)= -\frac{\partial }{\partial x}l_f(t,x), \ \ \ t\ge0,
\end{align}
subject to $l_f(t,0) = \nu_f(t), \ l_f(t, t/b)=0$ and $l_f(0,x) = \delta(x)$.

The next statement provides a characterization of the difference-differential equation solved by the probability mass function of a point process time-changed with the inverse of a Bernstein subordinator, $L_f$.

\begin{theorem}\label{teoremaProcessiPuntoConTempoInversoSubordinatore}
Let $N$ be a $d$-dimensional point process with probability mass function $p_n(t) = P\{N(t) = n\},\ n\in\mathbb{Z}^d,$ and $p_n^f$ the probability mass function of its time-changed version $N\circ L_f$. If Condition I holds and $p_n$ satisfies the difference-differential equation 
\begin{equation}\label{equazioneProcessoPuntoGenerale}
\frac{\dif }{\dif t}p_n(t) = O_d \,p_n(t),\ t\ge0, n\in\mathbb{Z}^d,\ \ \ p_n(0) = q(n),\ n\in\mathbb{Z}^d,
\end{equation}
with $q$ being a probability mass function over $\mathbb{Z}^d$ and $O_d$ being a difference operator such that $\forall \ n$,
\begin{align}\label{ipotesiOperatoreDifferenze}
 O_d\, p_n^f(t) = \int_0^\infty O_d\, p_n(x) l_f(t,x)\dif x,\ \forall\ t,
\end{align}
then, $p_n^f$  satisfies the equation
\begin{equation}\label{equazioneGeneraleDervitaConvoluzionaleProcessoPunto}
\mathcal{D}_t^f q_n(t) = O_d \,q_n(t), \ t\ge0, n\in\mathbb{Z}^d,\ \ \ q_n(0) = q(n),\ n\in\mathbb{Z}^d.
\end{equation}
\end{theorem}


\begin{proof}
First, we observe that
$$p_n^f(0) = \int_0^\infty  l(0,x)p_n(x)\dif x = \int_0^\infty \delta(x) p_n(x) \dif x = p_n(0) = q(n),$$
meaning that $p_n^f$ satisfies the initial condition in  \eqref{equazioneGeneraleDervitaConvoluzionaleProcessoPunto}. Now we show that it also satisfies the equation in \eqref{equazioneGeneraleDervitaConvoluzionaleProcessoPunto},
\begin{align}
\mathcal{D}_t^f p_n^f(t) &= \mathbb{D}_t^f p_n^f(t) - \nu_f(t)p_n^f(0) \nonumber\\
& = \int_0^\infty p_n(x) \mathbb{D}_t^f l_f(t,x)\dif x- \nu_f(t)p_n^f(0)\nonumber\\
& = - \int_0^\infty p_n(x) \frac{\partial }{\partial x}l_f(t,x)\dif x - \nu_f(t)p_n^f(0)\nonumber\\
& = p_n(0)l_f(t,0)  + \int_0^\infty \frac{\dif }{\dif x}p_n(x) l_f(t,x)\dif x - \nu_f(t)p_n^f(0) \nonumber\\
& = \nu_f(t) \big(p_n(0) - p_n^f(0)\big) +\int_0^\infty O_d\, p_n(x) l_f(t,x)\dif x\nonumber\\
& =  O_dp_n^f(x).\nonumber
\end{align}
Note that in the second equality we exchanged $\mathbb{D}_t^f p_n^f(t) =\int_0^\infty p_n(x) \mathbb{D}_t^f l_f(t,x)\dif x$ since $\frac{\dif}{\dif t} p_n^f(t) =\int_0^\infty p_n(x) \frac{\dif}{\dif t} l_f(t,x)\dif x$. In the first step we used \eqref{relazioneOperatoriConvoluzione}, in the third one \eqref{legameDerivataConvoluzionaleLiouvilleOperatoreSpaziale}, in the fourth one \eqref{equazioneProcessoPuntoGenerale} abd in the last one the initial conditions and \eqref{ipotesiOperatoreDifferenze}.
\end{proof}

From Theorem \ref{teoremaProcessiPuntoConTempoInversoSubordinatore} we derive  the following result concerning the vector processes studied in the previous sections.

\begin{corollary}
Let $L_f$ being the inverse of a subordinator with L\'{e}vy symbol a Bernstein function $f$ such that Condition I holds. 
\begin{itemize}
\item[($i$)] Let $(N_1,N_2)$ the homogeneous interacting Skellam vector process defined in \eqref{definizioneInfinitesimaleCoppiaTipoSkellam}, then the probability mass function of $(N_1\circ L_f,N_2\circ L_f)$ satisfies equation \eqref{equazioneStatoProcessoSkellamInterazione} with the time derivative replaced by the convolution-type derivative $\mathcal{D}_t^f$.

\item[($ii$)]  Let $(N_1,N_2)$ the homogeneous birth-death-migration vector process defined in \eqref{definizioneInfinitesimaleProcessoNascitaMorteInterazione}, then the probability mass function of $(N_1\circ L_f,N_2\circ L_f)$ satisfies equation \eqref{equazioneStatoProcessoNascitaMorteInterazione} with the time derivative replaced by the convolution-type derivative $\mathcal{D}_t^f$.
\end{itemize}
\end{corollary}

\begin{proof}
In the homogeneous case the difference operators appearing on the right-hand sides of \eqref{equazioneStatoProcessoSkellamInterazione} and \eqref{equazioneStatoProcessoNascitaMorteInterazione} do not depend on $t$ and they satisfy hypothesis \eqref{ipotesiOperatoreDifferenze}. Thus, by applying Theorem \ref{teoremaProcessiPuntoConTempoInversoSubordinatore} the proof completes.
\end{proof}

\begin{remark}
We point out that, in the case of a homogeneous Skellam-type vector process, the counting processes described in Remark \ref{remarkConteggioEventiSkellamInterazione} are renewal processes. This property is inherited by the time-changed interacting Skellam vector process. For instance the process of the type \eqref{processoConteggioVariazioniSkellamInterazione}, which counts the number of times that the time-changed vector modifies its state is a renewal process with i.i.d. waiting times with distribution, for $t\ge0$,
$$ P\{J>t\} = \mathbb{E}e^{-\bar \lambda L_f(t)} =  \tilde l(t, \bar\lambda),$$
where $\bar\lambda>0$ is the constant rate of the counting process in absence of the time-changed, that is a homogeneous Poisson process (since we are dealing with homogeneous cases only), see Remark \ref{remarkConteggioEventiSkellamInterazione}. We refer to \cite{K2011} for further details on this topic. \hfill $\diamond$
\end{remark}

\appendix

\section{Proof of equation \eqref{relazioneFormaAlternativaNumeroArmonico} }\label{appendiceDimostrazioneRelazioneFormaAlternativaNumeroArmonico}
Let $n\in \mathbb{N}$,
\begin{align*}
\sum_{k=1}^{n} \binom{n}{k} \frac{(-1)^k}{k} &=\frac{(-1)^n}{n} + \sum_{k=1}^{n-1} \binom{n-1}{k} \frac{(n \pm k)(-1)^k }{k(n-k)}\nonumber \\
& = \frac{(-1)^n}{n} + \sum_{k=1}^{n-1} \binom{n-1}{k} \frac{ (-1)^k }{k} +\frac{1}{n} \sum_{k=1}^{n-1} \binom{n}{k} (-1)^k \nonumber\\
& = \frac{(-1)^n}{n} +  \sum_{k=1}^{n-1} \binom{n-1}{k} \frac{ (-1)^k }{k} - \frac{1}{n}\big(1 + (-1)^n\big)\\
&= \sum_{k=1}^{n-1} \binom{n-1}{k} \frac{ (-1)^k }{k} -\frac{1}{n}
\end{align*}
and by iterating the reasoning we obtain the desired relationship,
$$ \sum_{k=1}^{n} \binom{n}{k} \frac{(-1)^k}{k} = - \sum_{k=1}^{n} \frac{ 1 }{k}  = -H_n, \ \ \ n\in\mathbb{N}.$$




\footnotesize{

}

\end{document}